\documentclass[11pt]{amsart}

\usepackage[utf8]{inputenc}
\usepackage[english]{babel} 

\usepackage{latexsym,amssymb, epsfig, amsmath,amsfonts, amsthm}
\usepackage{color}
\usepackage{multirow}
\usepackage{relsize}
\usepackage{microtype}
\usepackage{dsfont}
\usepackage{wasysym}
\usepackage{cases}
\usepackage{tikz}
\usepackage[colorlinks, allcolors=blue]{hyperref}
\usepackage{enumitem}
\usepackage[margin=1in]{geometry}

\numberwithin{equation}{section}
\newtheorem{theorem}{Theorem}[section]
\newtheorem{lemma}[theorem]{Lemma}

\newtheorem{corollary}[theorem]{Corollary}
\newtheorem{proposition}[theorem]{Proposition}
\newtheorem{remark}[theorem]{Remark}

\newtheorem{example}{Example}[section]

\newtheorem{assumption}{Assumption}

\newcommand{\ep}{\epsilon}

\newcommand{\fF}{{\mathfrak F}}
\newcommand{\fS}{{\mathfrak S} }

\newcommand{\dC}{\mathbb{C}}
\newcommand{\dD}{\mathbb{D}}
\newcommand{\dE}{\mathbb{E}}

\newcommand{\dP}{\mathbb{P}}  
\newcommand{\dR}{\mathbb{R}}
\newcommand{\dX}{\mathbb{X}}

\newcommand{\re}{\mathrm{e}}
\newcommand{\rd}{\mathrm{d}}

\newcommand{\cA}{\mathcal{A}}
\newcommand{\cB}{\mathcal{B}}
\newcommand{\cC}{\mathcal{C}}
\newcommand{\cF}{\mathcal{F}}
\newcommand{\cG}{\mathcal{G}}
\newcommand{\cH}{\mathcal{H}}
\newcommand{\cK}{\mathcal{K}}

\newcommand{\cP}{\mathcal{P}}
\newcommand{\cR}{\mathcal{R}}

\newcommand{\cZ}{\mathcal{Z}}

\newcommand{\wt}{\widetilde}

\newcommand{\baa}{\begin{eqnarray*}}
\newcommand{\eaa}{\end{eqnarray*}}

\newcommand{\indic}[1]{\mathds{1}_{#1}}

\newcommand{\ABS}[1]{{{\left| #1 \right|}}} 
\newcommand{\BRA}[1]{{{\left\{#1\right\}}}} 
\newcommand{\SCA}[1]{{{\left<#1\right>}}} 
\newcommand{\NRM}[1]{{{\left\| #1\right\|}}} 
\newcommand{\PAR}[1]{{{\left(#1\right)}}} 
\newcommand{\SBRA}[1]{{{\left[#1\right]}}} 
\newcommand{\TV}[1]{{{\| #1\|}_{\mbox{{\scriptsize TV}}}}} 

\newcommand{\ind}{\mathds{1}}

\providecommand{\keywords}[1]{\textbf{Keyswords: } #1}

\begin{document}

\title[Stochastic epidemic model with memory]{
A stochastic epidemic model with memory of the last infection and waning immunity}

\author[H. Guérin]{Hélène Guérin}
\thanks{HG acknowledges funding from the Natural Sciences and Engineering Research Council of Canada (NSERC) for its Discovery Grant (RGPIN-2020-07239)}
\address{Département de mathématiques, 
Université du Québec à Montréal}
\email{guerin.helene@uqam.ca}

    \author[A.B. Zotsa--Ngoufack]{Arsene Brice Zotsa--Ngoufack} 
\thanks{ABZN acknowledges funding from the SCOR Foundation for Science, and the FRQNT-CRM-CNRS support.}
\address{Département de mathématiques, 
Université du Québec à Montréal}
\email{zotsa\_ngoufack.arsene\_brice@uqam.ca}

\date{May 13, 2026}
  \begin{abstract}
We adapt the article of Forien, Pang, Pardoux and Zotsa \cite{forien-Zotsa2022stochastic}, on epidemic models with varying infectivity and waning immunity, to incorporate the memory of the last infection. 
 To this end, we introduce a parametric approach and consider a piecewise deterministic Markov process modeling both the evolution of the parameter, also called the trait, and the age of infection of individuals over time.
 At each new infection, a new trait is randomly chosen for the infected individual according to a Markov kernel, and their age is reset to zero.
In the large population limit, we derive a partial differential equation (PDE) that describes the density of traits and ages. The main goal is to study the conditions under which endemic equilibria exist for the deterministic PDE model and to establish an endemicity threshold that depends on the model parameters. The local stability of these equilibria is also analyzed. The endemicity threshold is computed for several examples, including models that incorporate a vaccination policy, and a local stability result is obtained for a memory-free SIS-type model.
 \end{abstract}

\maketitle

\section{Introduction}

We introduce a new approach to stochastic modeling in epidemiology, focusing on diseases whose infections do not confer full immunity. After an individual is first infected by a pathogen, T lymphocytes store information about the pathogen. Consequently, upon reinfection by the same pathogen, there is an immune response \cite{chaplin2010overview,chowdhury2020immune}. 
The model studied here is inspired by the stochastic epidemic model of type-SIS of
Forien, Pang, Pardoux and Zotsa-Ngoufack in \cite{forien-Zotsa2022stochastic} (see also \cite{ngoufack2024stochastic}). They consider a stochastic model that takes into account a varying infectivity and a waning immunity.  To this aim they introduce the susceptibility of an individual as the probability to be infected after a contact with an infected person.  However, they assume that at each new infection the new  infectivity and susceptibility does not depend on previous infections. In the present work, we adapt their individual-based epidemic model to take into account the fact that at each new infection the new infectivity and susceptiblity curves can depend on their last values. Besides, the model study here allows us to take into account a vaccination policy, which was not covered in the study of endemicity in \cite{forien-Zotsa2022stochastic}. 
\\
A disease is called endemic if it persists in a population over a long period of time. For public health purposes, it is of great importance to understand the mechanism that induces the existence of endemic states, and this question has been studied a lot in the literature \cite{britton2018stochastic,inaba_kermack_2001,inaba_variable_2017,magal2018theory,perthame2006transport,ThieYang,webb1985theory}. Vaccinating the population is one way to prevent the disease from becoming endemic. However, as we have seen with the COVID-19 crisis, vaccines do not confer total immunity for some diseases, but only slow the disease's spread.
The effect of a vaccination policy on the long time evolution of the endemicity for a similar model is in particular studied by Foutel-Rodier, Charpentier and Guérin \cite{foutel2023optimal}. The authors assume that at each event (infection, recovery (seen as the end of infectiousness), or vaccination) new curves of infection and susceptibility are randomly chosen independently of the past. Consequently, there is no dependence in their model between the individual's level of infectivity following an infection and the resulting level of immunity. They introduce a vaccination policy as a renewal process for susceptible individuals in the population. They obtain an explicit threshold in the large limit population depending on the susceptibility distribution of a typical individual in the population and the generic distribution between two vaccination times, such that when the basic reproduction number $R_0$ is greater than this threshold, there exists an endemic equilibrium. We recall that the reproduction number $R_0$ is frequently defined as the average number of infections produced by an infected individual in a population
completely vulnerable to the disease. Note that El Khalifi and Britton  also study in \cite{khalifi2022extending} the vaccination effect in a model with vaccination at a fixed time after recovery. We also mention that Heffernan and Keeling \cite{heffernan2009implications}  consider a deterministic epidemic model that captures the within-host dynamics of the pathogen and immune system, as well as the associated population-level transmission dynamics. They show, in the case of measles, how vaccination can have a range of unexpected consequences as it reduces the natural boosting of immunity and decreases the number of naive susceptibles. In fact, the immune response helps the body to react rapidly against a virus that has infected it in the past \cite{gray1993immunological}.
In the present work, the susceptibility of an individual may depend or not on the infectivity induced by the infection, and therefore generalized the model of \cite{foutel2023optimal}. However, contrary of \cite{khalifi2022extending,foutel2023optimal,heffernan2009implications}, we won't study here in details the effect of a vaccination policy on the long time behavior of the disease, we will only focus on the existence of an endemic equilibrium.

The model studied in this article is a sort of \emph{parametric model} based on \cite{forien-Zotsa2022stochastic}. We introduce a  non-negative $\sigma$-finite measure space $\PAR{\Theta,\cH,\nu}$, with $\Theta\subset \dR^d$ and $d\ge 1$.  We define the  \emph{age} of an individual as the duration since its last infection.
We consider a family of deterministic non-negative functions $(\lambda(\cdot,\theta), \gamma(\cdot,\theta))_{\theta\in\Theta}$ defined on $\mathbb{R}^+$, where $\lambda(a,\theta)$ and $\gamma(a,\theta)$ respectively  model the infectivity and the susceptibility at age $a$ of an individual with parameter $\theta$. Infectivity represents the virulence of an infection, i.e., the force of infection of an infected individual. Susceptibility represents the probability of being reinfected after contact with an infected individual: the lower the susceptibility, the stronger the individual's immunity to the disease. One may mention, for example, the article \cite{CauchemezFergusson2011} on parametric estimation of infectivity in discrete time.
We assume that the population is homogeneous, that all individuals behave in the same way to the disease, and that interactions between individuals are uniform within the population. Consequently, we do not consider cases where the disease does not affect all individuals in the same way. Demographic effects, such as the birth and death of individuals, are neglected in the present work. Incorporating demographic turnover would introduce a continuous influx of new susceptible and infected individuals into the population. This would translate into a non‑homogeneous source term in our limit PDE, thereby drastically increasing the technical complexity of the stability analysis. In this framework, the parameter $\theta$, referred to as the trait, does not model demographic or physiological heterogeneity within the population. Instead, it
represents the individual's phenotypic state, in particular the immunological signature resulting from past infections, including of the most recently encountered strain or the individual's current immune status. Thus, $\theta$ reflects the stochasticity of the infection process: upon reinfection, the individual's trait is updated to a new state $\wt\theta$ representing the acquisition of new immune memory. This formulation captures how past infections shape the individual's susceptibility profile, in line the biological motivations discussed in \cite{boni2008vaccination,gog2008impact}. Finally, we assume that the epidemic has spread long enough for all individuals to have been infected at least once at $t = 0$.

We introduce the measured space $(\dR_+\times\Theta, \cB(\dR_+)\otimes\cH, \rd a\otimes\nu)$, where $\rd a$ is the Lebesgue measure on $\dR_+$ and $\cB(\dR_+)$ is the borel $\sigma$-algebra on $\dR_+$.
Throughout the article, we assume that the infectivity $\lambda$ and the susceptibility $\gamma$ are
non-negative measurable functions on $\dR_+\times \Theta$ such that
\begin{itemize}
	\item  $\lambda$ is bounded by a constant $\lambda_*>0$;
\item $\gamma$ is  bounded by $1$.
\end{itemize}
Assume that $\gamma$ is bounded by $1$ is natural since it models the probability of being infected after a contact with an infected individual. The assumption on the infectivity curves has also been used in \cite{forien-Zotsa2022stochastic} to move from the individual-based model to a system of partial differential equations when the size of the population goes to infinity, as we will do here.

As in \cite{forien-Zotsa2022stochastic}, we allow the susceptibility to depend on the infectivity. Each time an individual is infected, a new value of the parameter $\theta$ is randomly chosen for that individual and their age drops to zero. The infectivity and susceptibility curves are therefore deterministic between two infections. In the present work, the choice of the new value of $\theta$ can depend on its previous value to keep the memory of the last infection. Depending on the epidemiological model, we could, for example, imagine that if an individual had a serious infection, the next could be less severe.
Consequently, the model studied in this paper belongs to the class of age and trait structured model. There is a wide literature on this type of model (see, e.g. \cite{crump1969general,crump1968general,hamza2013age,meleard2009trait,oelschlager1990limit,tran2006modeles}). We also mention that the model studied in \cite{meleard2009trait,tran2006modeles} is quite close to our model, but contrary to our case, the birth rate is independent of the state of other individuals and the birth process is independent of the death process.

Let us consider a population of size $N$. 
For $k\in\{1,\ldots, N\}$, we denote by $(a_k^N(t))_{t\ge 0}$ the age process and by $\theta_k^N(t)$ the parameter of the $k$-th individual at time $t$. Individuals interact through the force of infection in the population, denoted by $\fF^N$. This system is difficult to study for any finite $N$. However, for each $t\ge 0$ and for i.i.d initial values, when the size of the population goes to infinity, we prove that the empirical distribution of $(a_k^N(t), \theta_k^N(t))_{1\le k\le N}$ converges weakly to the solution of a nonlinear partial differential equation, similar to the one obtained by Kermack and McKendrick in \cite{kermack_contribution_1927,kermack_contributions_1932,kermack_contributions_1933} but without demographic effect. Note that a study of the fluctuations of the model can be found in \cite{zotsa2025functional}. 
Then we focus on the long time behaviour of this (deterministic) nonlinear equation. More precisely, under appropriate assumptions, we show that this equation admits a non-zero stationary solution and that there is therefore an endemic equilibrium. We obtain a very general endemicity threshold that we can compute for some examples in Section~\ref{sec-proofs-endemic-v}, including examples taking into account a vaccination policy in Section~\ref{sec:vaccin}.  We also study in Section~\ref{sec:local stability} the local stability of endemic equilibria. Unfortunately, it was not possible to obtain a complete proof for the model with memory, and the question remains open. Even for a model without memory, the study of the stability of equilibria is difficult to carry out because the usual techniques could not be applied. We could not conclude using Doeblin's argument, as was done in \cite{gabriel2018measure} for a conservative renewal equation in population dynamics, and in \cite{pakdaman2013relaxation,torres2022multiple} in the context of neuron populations, in particular because, for epidemiological models, we cannot assume that susceptibility and infectivity curves are bounded by below by a positive value. We use the tools of abstract semi-linear Cauchy problems (see \cite{magal2018theory,thieme1990semiflows,webb1985theory}) to obtain the local stability in a memory-free framework under semi-explicit
assumptions on infectivity and susceptibility curves.  
This is the only result of stability, to our knowledge, that goes this far.

\paragraph{\bf Notations}
For a measured space $\PAR{E,\cG,\mu}$,  $L^1(\mu)$ is the set of integrable functions with respect to the measure $\mu$, and more generally $L^p(\mu)$ with $p\in[1,\infty]$ is the Lebesgue space with respect to the measure $\mu$. For any measurable function $f$, non-negative or in $L^1(\mu)$, we denote $\SCA{\mu,f}=\int f\rd \mu$. The norm $\NRM{\cdot}_\infty$ is the classical uniform norm, and $\mathrm{ess}\sup_{\Theta}$ is the essential supremum on $\PAR{\Theta,\cH,\nu}$. When $\nu$ is a probability measure, for a non-negative or integrable with respect to $\nu$ measurable function $f$ defined on $\PAR{\dR_+\times\Theta,\cB(\dR_+)\otimes\cH}$, we define $\dE_\nu\SBRA{f(a)}=\int_\Theta f(a,\theta)\nu(\rd\theta)$ for any $a\in\dR_+$. For a function $f$ defined on $\dR_+$, we denote by $\mathrm{Supp}(f)$ its support.  We denote by $\dD(\dR_+,\cP(\dR_+\times\Theta))$ the Skorohod space of càdlàg functions on $\dR_+$ with values in the space $\cP(\dR_+\times\Theta)$ of probability measures on $\PAR{\dR_+\times\Theta,\cB(\dR_+)\otimes\cH}$.
Finally, $\cR e(\alpha)$ is the real part of a complex number $\alpha\in\dC$.

\subsection*{Organization of the article} The rest of the article is organized as follows. In Section \ref{sec-model-v}, we introduce the model. In Section \ref{sec:main-results}, we present the main results on the functional law of large numbers (FLLN) and on the long time behaviour of the disease. The proof of the FLLN is given in Section \ref{VVM-sec-FLLN} and the study of the existence of endemic equilibria and its local stability is presented in Section \ref{sec-proofs-endemic-v}. In Section \ref{sec:vaccin},  we study the existence of endemic equilibria for two explicit models taking into account a vaccination policy, including the one studied in \cite{foutel2023optimal}. 

\section{A stochastic parametric model with memory}\label{sec-model-v}
Let $\PAR{\Theta, \cH}$ be a measurable space, $\PAR{\Omega, \cF, \dP}$ be a probability space, and $\dE\SBRA{\cdot}$ be the expectation with respect to $\dP$. We consider a population of finite size $N$. 
For $k\in\{1,\ldots, N\}$, we denote by $(a_k^N(t))_{t\ge 0}$ the age process and by $\theta_k^N(t)$ the trait of the $k$-th individual at time $t$. We assume that $(a_0^k,\theta_0^k)_{1\le k\le N}$ are i.i.d random variables on $\PAR{\Omega, \cF, \dP}$ with distribution $\mu_0$ on $\dR_+\times\Theta$ modeling the initial age and parameter of each individual.

Each time an individual is infected, its age jumps to $0$ and a new parameter is randomly chosen according to a memory kernel $K$. The ages and parameters of the other individuals are not affected. Between two infections, the ages of all the individuals in the population increase linearly and their parameters remain constant.  

Throughout the paper, we work with a kernel satisfying following assumption.
\begin{assumption}\label{Hyp:Kernel}
$\nu$ is a non-negative $\sigma$-finite measure on $\Theta$ and $K:\Theta\times\Theta\to \dR_+$ is a  measurable function, such that for any $\theta\in\Theta\subset\mathbb{R}^d,$
$\displaystyle{\int_\Theta K(\theta,\widetilde\theta)\nu(\rd \widetilde\theta)=1}$.
\end{assumption}

Note that the kernel $K$ can be viewed as a function of the distance between strains or phenotypes $\theta$. It may model antigenic drift, acting as a mutation kernel through which new variants escape existing immunity. Alternatively, it may represent immune maturation, whereby reinfection shifts the host’s profile toward a more protected state and reduces susceptibility. Consequently, the choice of the kernel $K$ depends strongly on the characteristics and dynamics of the disease under consideration. For example, \cite{zeraick2024bridging} uses a gamma--Mittag-Leffler memory kernel over time. In the present work, however, $K$ is kept general, since the model is not tailored to a specific disease.

\smallskip

 Let us introduce $N$ independent Poisson measures $\PAR{Q_k}_{1\le k\le N}$ on $\mathbb{R}_+\times\Theta\times\mathbb{R}_+$ with intensity $\rd z\nu(\rd \theta)\rd t$.
The family  $(a^N_k,\theta^N_k)_{1\le k\le N}$ is then seen as the solution to the following system of stochastic differential equations:
\begin{align}\label{eq:SDE}
\begin{cases}
	a^N_k(t)=\displaystyle{a_0^k+t-\int_{0}^{t}\int_{\Theta\times\dR_+}a^N_k(s^-)\mathds{1}_{\fF^N(s^-)\gamma^N_k(s^-)K(\theta^N_k(s^-),\widetilde\theta)\ge z}Q_k\PAR{\rd z,\rd\widetilde\theta,\rd s}}\\[0.3cm]
	\theta^N_k(t)=\displaystyle{\theta_0^k+\int_{0}^{t}\int_{\Theta\times\dR_+}\hskip-0.2cm\left(\widetilde\theta-\theta^N_k(s^-)\right)\mathds{1}_{\fF^N(s^-)\gamma^N_k(s^-)K(\theta^N_k(s^-),\widetilde\theta)\ge z}Q_k\PAR{\rd z,\rd\widetilde\theta,\rd s}}\\[0.3cm]
	\gamma^N_k(t)=\gamma(a^N_k(t),\theta^N_k(t)),
	\end{cases}
\end{align}
where the force of infection in the population is given by
\begin{equation}
	\fF^N(t)=\frac{1}{N}\sum_{k=1}^{N}\lambda\PAR{a^N_k(t),\theta^N_k(t)}.
\end{equation}

We introduce the empirical measure $\mu ^N_t$ of ages and traits at time $t\ge 0$,
\begin{equation}\label{H-mu}
\mu ^N_t=\frac{1}{N}\sum_{k=1}^{N}\delta_{(a^N_k(t),\theta^N_k(t))}.
\end{equation}
Note that $\fF^N(t)=\langle\mu ^N_t,\lambda\rangle$.
We observe that individuals are in interaction through the force of infection of the disease in the population. 
In fact, the individual $k$ gets reinfected at time $t$ at a rate $\fF^N(t)\gamma^N_k(t)$ and if it occurs, their age jumps to $0$ and he is assigned a new parameter according to the distribution $K( \theta^N_k(t^-),\widetilde\theta)\nu(\rd \widetilde\theta)$.
The processes $(a_k^N,\theta_k^N)_{1\le k\le N}$ are constructed by induction on the jump times. The upper bound conditions on the curves $\lambda$ and $\gamma$ ensures that the rate of new infections, $\fF^N(t)\gamma_k^N(t)$, is bounded almost surely
by the constant $\lambda_\ast$, and therefore the jump times cannot accumulate.
The family $(a_k,\theta_k)_{1\le k\le N}$ is a system of interacting piecewise deterministic Markov processes on the Skorohod space $\dD(\dR_+, \dR_+\times \Theta)$. 
This stochastic system is well defined and has a unique solution under Assumption \ref{Hyp:Kernel} (see, e.g. \cite{davis1984}).

\section{Main results}\label{sec:main-results}

The long-term behavior of the system \eqref{eq:SDE} is difficult to study  due to the interactions between individuals. However, for each $t\ge 0$ and for i.i.d. initial values, when the size of the population goes to infinity, the empirical distribution $\mu ^N_t$ of ages and traits, defined in \eqref{H-mu}, converges weakly to the solution of a nonlinear equation. The system being exchangeable, there is propagation of chaos. 
\begin{theorem}\label{VVM-th}
Let $\lambda$ and $\gamma$ be non-negative measurable bounded functions respectively by $\lambda_*$ and $1$. Under Assumption~\ref{Hyp:Kernel},  as $N\to\infty,\,\mu ^N$ converges in law to a measure $\mu \in\dD(\dR_+;\mathcal{P}(\dR_+\times\Theta))$, which is the unique solution to 
	\begin{equation}\label{VVM-eq1}
	\langle\mu _t,f_t\rangle=\langle\mu _0,f_0\rangle + \int_0^t \langle\mu _s,\partial_af_s+\partial_sf_s\rangle  \rd s+\int_{0}^{t}\langle\mu _{s},\lambda\rangle\langle\mu _{s},Rf_{s}\rangle\rd s,
	\end{equation}
	for any measurable bounded function $f$ on $\dR_+\times\dR_+\times\Theta$, of class $\cC^1$ with respect to its first two variables, where the operator $R$ is given by
		\begin{equation}\label{eq:Rjump-term}
	Rf(a,\theta)=\int_{\Theta}\left(f(0,\widetilde\theta)-f(a,\theta)\right)\gamma( a,\theta)K(\theta,\widetilde\theta)\nu(\rd\widetilde\theta).
	\end{equation}
\end{theorem}
This theorem is proved in Section~\ref{VVM-sec-FLLN}. 
When $\mu_0$ has a density $u_0$ with respect to the measure $\rd a\nu(\rd\theta)$ on $\dR_+\times \Theta$, the weak solution $\mu_t$ of \eqref{VVM-eq1} also admits a density, denoted by $u_t$ (see Proposition~\ref{AP-density}). We easily deduce that $(u_t)_{t\ge 0}$ is then a weak solution to the following partial differential equation (PDE): $\forall(a,\theta)\in\dR_+\times\Theta$, 
	\begin{numcases}{}
	\partial_t u_t(a,\theta)+\partial_a u_t(a,\theta)=-\fF(t)\gamma(a,\theta)u_t(a,\theta)\nonumber\\[0.3cm]
	u_t(0,\theta)=\displaystyle{\fF(t)\int_{\dR_+\times\Theta}\gamma(a,\widetilde\theta)K(\widetilde\theta,\theta)u_t(a,\widetilde\theta)\rd a\nu(\rd\widetilde{\theta})}\label{VVL-eq-1}\\[0.3cm]
	u(0,a,\theta)=u_0(a,\theta)\nonumber\\[0.3cm]
	\fF(t)=\displaystyle{\int_{\dR_+\times\Theta}\lambda(a,\theta)u_t( a,\theta)\rd a\nu(\rd\theta).}\label{eq:def-F}
	\end{numcases}

	We introduce the deterministic function $\fS$ defined on $\dR_+\times \Theta$ by
	\begin{align}
	\fS (t,\theta)&=\int_{\dR_+\times\Theta}\gamma(a,\widetilde\theta)K(\widetilde\theta,\theta)u_t(a,\widetilde\theta)\rd a\nu(\rd\widetilde{\theta})\label{eq:def-S},
	\end{align}
    which denotes the average susceptibility of the population with trait $\theta$ at time $t$.
    Our main goal is to study the long time behaviour of the solution to the deterministic nonlinear equation~\eqref{VVL-eq-1}. 
First, using the method of characteristics, we easily obtain the following way of writing the solution.
\begin{proposition}\label{Th-equiv-FPPZ}
Given a solution $u$ to the PDE  \eqref{VVL-eq-1}, then the pair $(\fF,\fS)$, defined by \eqref{eq:def-F} and \eqref{eq:def-S} respectively, is the unique solution to the following system of integral equations  
\begin{equation}\label{equiv-s-F}
		\begin{cases}
		\displaystyle{\fS (t,\theta)=\int_{0}^{t}\int_{\Theta}\gamma(t-a,\widetilde\theta)\re^{-\int_{a}^{t}\fF(s)\gamma(s-a,\widetilde\theta)\rd s}\fF(a)\fS ( a,\widetilde\theta)K(\widetilde\theta,\theta)\rd a\nu(\rd\widetilde\theta)}\\[0.3cm]
		\hspace{1.5cm}\displaystyle{+\int_{0}^\infty\int_{\Theta}\gamma(a+t,\widetilde\theta) \re^{-\int_{0}^{t}\fF(s)\gamma(a+s,\widetilde\theta)\rd s}u_0(a,\widetilde\theta)K(\widetilde\theta,\theta)\rd a\nu(\rd\widetilde\theta)}\\[0.3cm]
		\displaystyle{\fF(t)=\int_{0}^{t}\int_{\Theta}\lambda(t-a,\widetilde\theta)\re^{-\int_{a}^{t}\fF(s)\gamma(s-a,\widetilde\theta)\rd s}\fS (a,\widetilde\theta)\fF(a)\rd a\nu(\rd\widetilde\theta)}\\[0.3cm]
		\hspace{1.5cm}\displaystyle{+\int_{0}^\infty\int_{\Theta}\lambda(a+t,\widetilde\theta) \re^{-\int_{0}^{t}\fF(s)\gamma(a+s,\widetilde\theta)\rd s}u_0(a,\widetilde\theta)\rd a\nu(\rd \widetilde\theta).}
		\end{cases}
	\end{equation}
Conversely, given $(\fF,\fS)$ the solution to the system~\eqref{equiv-s-F}, the function
\begin{equation}\label{VVL-eq-2'}
	  u_t(a,\theta)=\begin{cases}
	  u_0(a-t,\theta)\exp\left(-\int_{0}^{t}\fF(s)\gamma(a-t+s,\theta)\rd s\right)&\mbox{ if }a> t\\
	  \fF(t-a)\fS (t-a,\theta)\exp\left(-\int_{t-a}^{t}\fF(s)\gamma(s-t+a,\theta)\rd s\right)&\mbox{ if }a\le t.
	  \end{cases}
	  \end{equation}
 is the unique solution to the system~\eqref{VVL-eq-1}.
\end{proposition}
Using the fact that $u_t$ is a probability density function with respect to $\rd a\otimes\nu$, we get the following straightforward consequence.

\begin{remark}\label{rk:borne de F et S}
We note that $\fF$ and $\fS$ satisfy for any $t\ge 0$
\begin{align}\label{VVM-eq-weight}
	&\int_{0}^\infty\int_{\Theta}\exp\left(-\int_{0}^{t}\fF(s)\gamma(a+s,\theta)\rd s \right) u_0(a,\theta)\rd a\nu(\rd\theta)\\
&\hspace{2cm}+\int_{0}^{t}\int_{\Theta}\fF(a)\fS(a,\theta)\exp\left(-\int_{a}^{t}\fF(s)\gamma(s-a,\theta)\rd s \right)\rd a\nu(\rd\theta)=1.\nonumber
\end{align}
Consequently, using Expression~\eqref{equiv-s-F}, the upper bounds of the curves $\lambda$ and $\gamma$, and Assumption~\ref{Hyp:Kernel},  we have for any $t\ge 0$, $\fF(t)\le \lambda_*$ and   $\int_\Theta\fS(t,\theta)\nu(\rd \theta)\le 1$.
\end{remark}

\begin{remark}\label{Rq-equiv-FPPZ}
When there is no memory of the previous infection, i.e. $K(\theta,\wt\theta)=K(\wt\theta)$ does not depend on $\theta$, we recover the result of Forien et al.: the system  \eqref{equiv-s-F} satisfied by $(\fF,\fS)$ is identical to the system \cite[Equations (3.7)-(3.8)]{forien-Zotsa2022stochastic}. In addition, when $\lambda,\gamma$ are indicator functions whose supports depend on $\theta$, we recover SI(R)S-type compartmental models, as detailed in \cite[Section 2.1]{foutel2023optimal}.
\end{remark}

Before studying the long-term behavior of the solution to \eqref{VVL-eq-1}, we first identify its equilibria. To do this, we need to make some additional assumptions, which will provide a framework for the type of epidemiological model addressed in our study.

\begin{assumption}\label{Hyp:endemicity} 
\begin{enumerate}
    \item \label{Hyp:infectivity and susceptibility}
    $\lambda$ and $\gamma$ are non-negative measurable functions on $\dR_+\times \Theta$ such that
\begin{enumerate}
\item $\lambda$ is bounded by a constant $\lambda_*>0$;
\item $\gamma$ is bounded by $1$, with $\gamma(0,\cdot)\equiv 0$;
\item\label{Hyp:disjoint-support} $\lambda\gamma\equiv0$, and $\forall\theta\in\Theta$,  
$\sup\{t\ge0,\,\lambda(t,\theta)>0\}\le\inf\{t\ge0,\,\gamma(t,\theta)>0\}.$
\end{enumerate}
\item \label{Hyp:mean-susceptibility} There exists a positive measurable function $\gamma_*:\Theta\to[0,1]$ 
such that
	\[
	\lim_{a \to +\infty}\frac{1}{a}\int_0^a\gamma(s,\cdot)\rd s=\gamma_*(\cdot)\quad \nu\text{-a.e.}
	\]
\item \label{Hyp:ess sup}  $\forall x>0$,
		$\displaystyle{\mathrm{ess}\sup_{\Theta}\int_{\dR_+}\exp\left(-x\int_{0}^{a}\gamma(s,\cdot)\rd s\right)\rd a<\infty.}$
\item \label{Hyp:kernel in L^1} The kernel $K$ is positive on $\Theta^2$ and   $\displaystyle{\theta\mapsto \sup_{\widetilde\theta\in\Theta}K(\widetilde\theta,\theta)}$ belongs to $L^1(\nu)$.
\end{enumerate}
\end{assumption} 

 Assumption~\ref{Hyp:endemicity}.\ref{Hyp:infectivity and susceptibility} means that the susceptibility jumps to $0$ after an infection, which prevents an individual from being reinfected immediately after an infection, and 
 as long as an individual remains infectious, its susceptibility is equal to $0$, so they cannot be re-infected. These hypotheses are usual in epidemiological models and have also been used in \cite{forien-Zotsa2022stochastic,foutel2023optimal,zotsa2023stochastic}. 
The existence of a positive force of infection at equilibrium relies heavily on the long time behavior of the susceptibility given in Assumption~\ref{Hyp:endemicity}.\ref{Hyp:mean-susceptibility}, and on the primarily technical Assumptions~\ref{Hyp:endemicity}.\ref{Hyp:ess sup} and \ref{Hyp:endemicity}.\ref{Hyp:kernel in L^1}.
\begin{remark}
\begin{enumerate}
\item When $\gamma_*$ is bounded from below by a positive constant $\nu$-a.e, then Assumption~\ref{Hyp:endemicity}.\ref{Hyp:ess sup} is satisfied.
\item
When the susceptibility curves satisfy one of the following conditions 
\begin{itemize}
    \item $a\mapsto \gamma(a,\cdot)$ are non-decreasing and non-null functions $\nu$-a.e, as assumed in \cite[Assumption 4.1]{forien-Zotsa2022stochastic}, or 
    \item $\gamma(a,\cdot)$ has a  positive limit $\gamma_*(\cdot)$ 
    when $a\to +\infty$ $\nu$-a.e, or
    \item $a\mapsto \gamma(a,\cdot)$ are periodic functions $\nu$-a.e,
\end{itemize} 
we easily observe that Assumption~\ref{Hyp:endemicity}.\ref{Hyp:mean-susceptibility} is satisfied.
\end{enumerate}
\end{remark}

In the following theorem, proved in Section~\ref{sec-proofs-endemic-v-ex}, a threshold of existence of an endemic equilibrium is identified. In what follows, we denote by $T_I(\theta)$ the infection period of an individual of trait $\theta$: $T_I(\theta)=\sup\BRA{t\ge 0: \forall a>t, \lambda(a,\theta)>0}$.
\begin{theorem}\label{thm:existence-endemicity}
    Under Assumptions~\ref{Hyp:Kernel} and \ref{Hyp:endemicity}, there is a unique function $\fS _*$ in $L^1(\nu)$, positive $\nu$-a.e., solution to
\begin{equation*}
\fS _*(\theta)=\int_{\Theta}K(\wt\theta,\theta)\fS _*(\wt\theta)\nu(\rd\wt\theta)\text{ with }\int_{\dR_+\times\Theta}\lambda(a,\theta)\fS _*(\theta)\rd a\nu(\rd\theta)=1.
  \end{equation*}
Moreover, there exists an endemic equilibrium when 
    \begin{equation} \label{eq:condition-endemic}
    \int_{\Theta}\frac{1}{\gamma_*(\theta)}\fS _*(\theta)\nu(\rd \theta)<1.
    \end{equation}
    In this case, the {\it endemic} proportion of infected individuals is given by \begin{equation}
        I_*=\fF_*\int_{\Theta} T_I(\theta)\fS_*(\theta)\nu(\rd\theta),
    \end{equation}
    where $\fF_*$ is the force of infection at the equilibrium, as defined in Lemma~\ref{VVM-lem-ex}.
    
    This equilibrium is unique if we also assume that $\nu$-a.e,
    \begin{equation} \label{eq:condi-uniq-endemic}
    \forall a\ge 0 \quad \gamma(a,.)\ge \frac{1}{a}\int_0^a\gamma(s,.)\rd s.
    \end{equation}

 Under condition \eqref{eq:condi-uniq-endemic}, when $\displaystyle{\int_{\Theta}\frac{1}{\gamma_*(\theta)}\fS _*(\theta)\nu(\rd \theta)>1}$, there is no endemic equilibrium, the only equilibrium is disease free.

\end{theorem}

\begin{remark}\label{rk:endemic-conditon}
Since $\displaystyle{\int_{\dR_+\times\Theta}\lambda(a,\theta)\fS _*(\theta)\rd a\nu(\rd\theta)=1}$,
 the condition \eqref{eq:condition-endemic} in Theorem~\ref{thm:existence-endemicity} for the existence of an endemic equilibrium can be written
    \begin{equation}\label{eq:R_0^*}
    \dE_\nu^*\SBRA{\frac{1}{\gamma_*}}<R_0^*\quad \text{with }R_0^*=\dE_\nu^*\SBRA{\int_0^\infty\lambda(a)\rd a},
    \end{equation}
    where $\dE_\nu^*$ is the expectation with respect to the probability measure $\kappa^{-1}\fS _*(\theta)\nu(\rd\theta)$  on $\Theta$, with $\kappa=\int_\Theta \fS _*(\theta)\nu(\rd\theta)$. $R_0^*$ can thus be seen as the average number of infections, under this new probability measure, produced by an infected individual in a population
completely vulnerable to the disease.
    Consequently, up to a change of measure, we obtain the same kind of condition for the existence of an endemic equilibrium as in \cite{forien-Zotsa2022stochastic}.
\end{remark}
In the memory-free case, Theorem~\ref{thm:existence-endemicity} extends the result of \cite{forien-Zotsa2022stochastic} to non-monotone susceptibility curves (see Example~\ref{ex:values of S_*}.\ref{ex:no memory}), and also extends the result of \cite{foutel2023optimal} to non-independent infectivity and susceptibility curves, in the case of a \emph{parametric model}. In particular, when there is no memory of the last infection, we recover in Section~\ref{sec:vaccin} the same threshold as in \cite{foutel2023optimal}, revealing that their assumption of independence between the infectivity and the susceptibility curves is not necessary. We also describe precisely in Section~\ref{sec:vaccin} the long time behavior of the disease for a \emph{toy model} with vaccination and with memory of the last infection, thereby highlighting in some cases the non-uniqueness of endemic equilibria when condition \eqref{eq:condi-uniq-endemic} is not satisfied.

\medskip
We derive the local stability of an endemic equilibrium when there is no memory of the last infection and
 under the following assumption, which means that the susceptibility curves are uniformly positive in long time.
\begin{assumption}\label{Hyp:endemic-stability}
    There exist  $\sigma\in(0,1]$ and a positive constant $a_*$ such that  
    \[
    \gamma(a,\theta)\ge \sigma\ind_{[a_*,+\infty)}(a), \quad \forall (a,\theta)\in\dR_+\times \Theta.
    \]
\end{assumption}
Assuming that $a_*$ does not depend on $\theta$ is a strong assumption, as it implies that the duration of infectivity, namely $T_I(\cdot)$, is deterministically bounded and excludes the possibility of modeling infectivity durations with, for example, an exponential distribution.
When there is no memory, we notice in  Example~\ref{ex:values of S_*} that $\fS_*\equiv\frac{1}{R_0}$  with $R_0=\dE_\nu\SBRA{\int_0^\infty\lambda(a)\rd a}$.
For technical reasons, we also need to assume that the measure $\nu$ to be absolutely continuous with respect to the Lebesgue measure.

\begin{theorem}\label{VVM-conj}
Assume that $\Theta$ is an open subset of $\dR^d$. 
 We assume that there is no memory of the previous infections: $K\equiv 1$ and $\nu$ is a probability measure  on $\Theta$ absolutely continuous with respect to the Lebesgue measure. Under Assumptions~\ref{Hyp:endemicity}, \ref{Hyp:endemic-stability}, and under the endemicity condition
 $
    \displaystyle{R_0>\int_\Theta \frac{1}{\gamma_*(\theta)}\nu(\rd \theta)}$,
    we denote $u_*$ an endemic equilibrium of the PDE~\eqref{VVL-eq-1}. 
If $\lambda$ and $\gamma$ are such that Equation~\eqref{eq:condition-sur-alpha} has no solution $\alpha\in\mathbb{C}$ with $\mathcal{R}e (\alpha)\ge 0$, then there is local stability of the equilibrium. 
    More precisely, there exists $w_0<0$ such that for any $w\in(w_0,0)$,  
    there exist $\delta,c>0$ such that $\NRM{u_0-u_*}\le \delta$ implies that for any $t\ge 0$,
    \[
    \NRM{u_t-u_*}_{ L^1(\rd a\otimes\nu)}\le c\re^{wt}\NRM{u_0-u_*}_{ L^1(\rd a\otimes\nu)},
    \]
where $u$ is the solution to PDE~\eqref{VVL-eq-1} starting from $u_0$.
\end{theorem}

A particular SIS-type model, satisfying assumptions of Theorem~\ref{VVM-conj}, is presented in Section~\ref{end-spec-L}.
This theorem is proved in Section~\ref{sec:local stability}. As mentioned in the introduction, classical techniques cannot be applied to prove Theorem~\ref{VVM-conj}. Our model is not covered by the work of \cite{magal2018theory,thieme1990semiflows,webb1985theory}. 
Indeed, the  study in \cite[Chapter~$8$, Section~$8.2.2$]{magal2018theory} requires the susceptibility curves $\gamma$ to be strictly positive, which is an unusual hypothesis in epidemiology because it prevents to have a full immunity period. Furthermore, the boundary condition of our PDE~\eqref{VVL-eq-1} is different than the one in \cite{thieme1990semiflows,webb1985theory}. See also \cite[Section I.1.2]{ngoufack2024stochastic} for more details.

\section{Functional law of Large Numbers}\label{VVM-sec-FLLN}

In this section, we prove Theorem~\ref{VVM-th}. We recall that we consider an homogeneous population of size $N\ge2$.\\
Let $(a_0^k,\theta_0^k)_{1\le k\le N}$ be i.i.d. random variables with distribution $\mu_0$ on $\dR_+\times\Theta$ defined on the probability space $\PAR{\Omega,\cF,\dP}$, modeling the initial age and parameter of each individuals in the population.  We consider $(a^N_k,\theta^N_k)_{1\le k\le N}$ the solution to the SDEs~\eqref{eq:SDE}. We study here the convergence of the associated (random) empirical measure $\mu^N$, defined by \eqref{H-mu}, to a deterministic measure $\mu$ on $\dR_+\times \Theta$ when $N\to\infty$.

\medskip
We easily check that for any test functions $f:\dR_+\times\dR_+\times\dR^d\to\dR_+$, such that for all $\theta\in\Theta,\, (t,a)\mapsto f_t(a,\theta)=f(t,a,\theta)$ is a continuously differentiable function with respect to its first two variables, we have
\begin{multline}\label{VV-eq-m}
	\langle\mu ^N_t,f_t\rangle
=\langle\mu ^N_0,f_0\rangle+\int_0^t \langle\mu ^N_s,\partial_af_s+\partial_sf_s\rangle  \rd s +\int_{0}^{t}\langle\mu ^N_{s},\lambda\rangle\langle\mu ^N_{s},Rf_{s}\rangle\rd s
\\
+\frac{1}{N}\sum_{k=1}^N\int_{0}^{t}\hskip-0.1cm\int_{\Theta\times\dR_+}\left(f_{s}(0,\widetilde\theta)-f_s(a^N_k(s^-),\theta^N_k(s^-))\right) \\
\times\mathds{1}_{\fF^N(s^-)\gamma^N_k(s^-)K(\theta^N_k(s^-),\widetilde\theta)\ge z}\overline{Q}_k(\rd z,\rd\widetilde \theta,\rd s),
\end{multline}
	where $\overline{Q}_k$ is the compensated Poisson measure of $Q_k$, and $Rf$ is defined by \eqref{eq:Rjump-term}.

	\subsection{The deterministic limit}
	Recall that $\mu_0$ is the initial distribution of ages and parameters in the population. 
		\begin{remark}
		Note that when $(a_0^k,\theta_0^k)_{1\le k\le N}$ are i.i.d. random variables with distribution $\mu_0$, by the law of large number, it follows that, the random measure $\mu _0^N$ converges to $\mu_0$ in law a.s., when $N\to\infty$. 
	\end{remark}
	We denote by $\mathcal{P}(\dR_+\times\Theta)$ the space of probability measures on $\PAR{\dR_+\times \Theta, \cB(\dR_+)\otimes\cH}$.
	If $\mu^N$ converges weakly to a probability measure $\mu =\left(\mu _t\right)_{t\ge0}\in\dD(\dR_+,\mathcal{P}(\dR_+\times\Theta))$, by \eqref{VV-eq-m}, $\mu$ should then satisfy  Equation~\eqref{VVM-eq1}
	for any measurable bounded function $f$ on $\dR_+\times\dR_+\times\Theta$, and of class $\cC^1$ with respect to its first two variables. 
We now introduce the norm in total variation on $\mathcal{P}(\dR_+\times\Theta)$, defined by
	\[ \TV{\mu-\nu}=\sup_{\varphi\in L^\infty(\rd a\otimes \nu),\,\|\varphi\|_{\infty}\le1}\left|\langle \mu-\nu,\varphi\rangle\right|.\]

\begin{proposition}\label{VVM-ex-th}
	Under the assumptions of Theorem~\ref{VVM-th}, the solution $\mu =\left(\mu _t\right)_{t\ge0}\in\dD(\dR_+;\mathcal{P}(\dR_+\times\Theta))$ to Equation~\eqref{VVM-eq1}  is unique. 
\end{proposition}
\begin{proof}
Assume that there are two solutions $\mu ^1$ and $\mu ^2$ of Equation~\eqref{VVM-eq1} with the same initial condition $\mu _0$. Let $\varphi$ be a test function, in the sense that $\varphi$ is a measurable bounded function on $\dR_+\times \Theta$, of class $\cC^1$ with respect to its first variable, with $\|\varphi\|_{\infty}\le1$. 
		
For any fixed $t>0$ and for all $\theta\in\Theta$, the following parametric transport equation
\begin{equation*}
	\begin{cases}
		\partial_s f_s(a,\theta)+\partial_a f_s(a,\theta)=0\quad\forall s\in[0,t]\\
		f_t(a,\theta)=\varphi(a,\theta),
	\end{cases}
\end{equation*}
has a unique solution $f:(s,a,\theta)\to f_s(a,\theta)$ defined by:
\begin{equation*}
\forall t\in\dR_+,\forall s\in[0,t],\forall (a,\theta)\in\dR_+\times\Theta,\quad f_s(a,\theta)=\varphi(a-(s-t),\theta).
\end{equation*}
Consequently, from  \eqref{VVM-eq1} and \eqref{eq:Rjump-term}, for $i\in\BRA{1,2}$,
\begin{equation}
	 \langle\mu _t^i,\varphi\rangle=\langle\mu _0,\varphi_t\rangle+\int_{0}^{t}\langle\mu _{s}^i,\lambda\rangle\langle\mu _{s}^i,R\varphi_{t-s}\rangle\rd s,
\end{equation}
where $\varphi_s(a,\theta)=\varphi( a+s,\theta)$. Since $\NRM{\varphi}_{\infty}\le 1$, $\gamma\in[0,1]$, and by Assumption~\ref{Hyp:Kernel}, we note that $\left|R\varphi_{t-s}(a,\theta)\right|\le 2$.
Therefore, using again the upper bounds of $\gamma$ and $\lambda$, it follows that,  for $t\ge0$  
\begin{align*}
	\left|\langle \mu ^1_t-\mu ^2_t,\varphi\rangle\right|&\le\int_{0}^{t}\left|\langle\mu _{s}^1-\mu ^2_s,\lambda\rangle\right|\left|\langle\mu _{s}^1,R\varphi_{t-s}\rangle\right|\rd s+\int_{0}^{t}\left|\langle\mu _{s}^2,\lambda\rangle\right|\ABS{\langle\mu _{s}^1-\mu _s^2,R\varphi_{t-s}\rangle}\rd s\\
	&\le 4\lambda_*\int_{0}^{t}\TV{\mu _s^1-\mu _s^2} \rd s.
\end{align*}	
Since this class of test functions is dense in $\BRA{\varphi \in L^\infty(\rd a\otimes\nu): \NRM{\varphi}_{L^\infty(\rd a\otimes\nu)}\le 1}$, we deduce that 
\[\sup_{r\in[0,t]}\TV{\mu _r^1-\mu _r^2}\le 4\lambda_*\int_{0}^{t}\sup_{0\le r\le s}\TV{\mu _r^1-\mu _r^2}\rd s,\]
and by Gronwall's lemma we obtain $\sup_{r\in[0,s]}\TV{\mu _r^1-\mu _r^2}=0$ $\forall t>0$. Then the uniqueness is proved.

\end{proof}

\subsection{Propagation of chaos}

In this section, we prove Theorem~\ref{VVM-th}. First, we make a few observations.
From Proposition~\ref{VVM-ex-th} and Theorem~\ref{VVM-th}, we have the following straightforward corollary.
\begin{corollary}
Under the assumptions of Theorem~\ref{VVM-th}, there exists a unique solution to Equation~\eqref{VVM-eq1}.
\end{corollary}

\begin{remark}
As we will see in the proof of Theorem~\ref{VVM-th}, the solution $\mu$ to Equation~\eqref{VVM-eq1} is the law of  a couple $(a(t),\theta(t))_{t\ge 0}$ of random processes  solution to the following nonlinear stochastic differential system
\begin{align*}
\begin{cases}
\displaystyle{a(t)=a_0+t-\int_{0}^{t}\int_{\Theta\times\dR_+}a(s^-)\mathds{1}_{\fF(s^-)\gamma(s^-)K(\theta(s^-),\widetilde\theta)\ge z}Q(\rd z,\rd \widetilde\theta,\rd s)}\\[0.3cm]
\displaystyle{
\theta(t)=\theta_0+\int_{0}^{t}\int_{\Theta\times\dR_+}\left(\widetilde\theta-\theta(s^-)\right)\mathds{1}_{\fF(s^-)\gamma(s^-)K(\theta(s^-),\widetilde\theta)\ge z}Q(\rd z,\rd\widetilde\theta,\rd s)}\\[0.3cm]
\gamma(t)=\gamma(a(t),\theta(t)),
\end{cases}
\end{align*}
where $(a_0,\theta_0)$ is a random variable with distribution $\mu_0$, $\fF$ is defined by \eqref{eq:def-F}, and $Q$ is a Poisson measure  on $\mathbb{R}_+\times\Theta\times\mathbb{R}_+$ with intensity $\rd z\nu(\rd \widetilde\theta)\rd t$, and independent of $(a_0,\theta_0)$. We also note that $\fF(t)=\dE\left[\lambda(a(t),\theta(t))\right]$.
\end{remark}

To prove Theorem~\ref{VVM-th}, we first prove tightness  of the sequence $\PAR{\mu^N}_{N\ge 2}$, then identify the limits as a solution to Equation~\eqref{VVM-eq1}. By uniqueness of the solution to Equation~\eqref{VVM-eq1}, the convergence of $\PAR{\mu^N}_{N\ge 2}$ to $\mu$ on any time interval $[0,T]$ follows.
\\
For the tightness of $\PAR{\mu^N}_{N\ge 2}$, we  use the following criterion,  see \cite[Lemma~$3.1$]{pang2022-CLT-functional}.
\begin{lemma}\label{VVM-Lem-20}
	Let $\PAR{X^N}_{N\ge 2}$ be a sequence of random processes taking values in the Skorohod space $\dD(\dR_+, \dR^d)$ such that $X^N(0)=0$.
	If for all $T>0$, $\ep>0$,
	\begin{align*}
\lim_{\delta\to0}\limsup_{N\to\infty}\sup_{0\le t\le T}\frac{1}{\delta}\dP\bigg(\sup_{0\le r\le \delta}|X^N(t+r)-X^N(t)|>\ep\bigg)=0,
	\end{align*}
	then the sequence $X^N$ is tight in $\dD(\dR_+, \dR^d)$.
\end{lemma}

\begin{lemma}\label{VVM-lem-tight}
	Under the assumptions of Theorem~\ref{VVM-th}, the sequence $\PAR{\mu ^N}_{N\ge 2}$ defined by \eqref{H-mu} is tight in $\dD\left(\dR_+,\mathcal{P}(\dR_+\times\dR^d)\right)$.
\end{lemma}

\begin{proof}
Let $\cC_0(\dR_+\times\dR^d)$ be the space of continuous real-valued function converging to $0$ at infinity, with the uniform norm $\NRM{f}_\infty=\sup_{a\ge 0, \theta\in\dR^d}\ABS{f(a,\theta)}$.
Fix $T>0$.
	From \cite[Theorem $2.1$]{roelly-coppolettaCriterionConvergenceMeasurevalued1986}, it suffices to prove that $\langle\mu ^N,f\rangle$ is tight in $\dD([0,T], \dR)$ for any $f\in\cC_0(\dR_+\times\dR^d)$ with a derivative with respect to its first variable and $\NRM{\partial_af}_{\infty}<+\infty$, which is a dense family in $\cC_0(\dR_+\times\dR^d)$.  
    
    Since $f$ and $\partial_af$ are bounded, we have,  for all $0\le s\le t\le T$, 
	\begin{align*}
	&\left|\langle\mu ^N_t,f\rangle-\langle\mu ^N_s,f\rangle\right|\\
	&\le\int_s^t \langle\mu ^N_r,\partial_af\rangle  \rd r\,+\\
    &\frac{1}{N}\sum_{k=1}^N\hskip-.1cm\int_{s}^{t}\hskip-0.1cm\int_{\Theta\times\dR_+}\hskip-0.4cm|f(0,\widetilde\theta)-f(a^N_k(r^-),\theta^N_k(r^-))|\mathds{1}_{\fF^N(r^-)\gamma^N_k(r^-)K(\theta^N_k(r^-),\widetilde\theta)\ge z}Q_k\PAR{\rd z,\rd\widetilde \theta,\rd r}\\
	&\le\|\partial_af\|_\infty|t-s|+\frac{2\|f\|_\infty}{N}\sum_{k=1}^N\int_{s}^{t}\int_{\Theta\times\dR_+}\mathds{1}_{\fF^N(r^-)\gamma^N_k(r^-)K(\theta^N_k(r^-),\widetilde\theta)\ge z}Q_k\PAR{\rd z,\rd\widetilde \theta,\rd r}\\
	&\le\left(\|\partial_af\|_\infty+2\lambda_*\|f\|_\infty\right)|t-s|\,+\\
    &\hskip 1cm\frac{2\|f\|_\infty}{N}\sum_{k=1}^N\int_{s}^{t}\int_{\Theta\times\dR_+} \mathds{1}_{\fF^N(r^-)\gamma^N_k(r^-)K(\theta^N_k(r^-),\widetilde\theta)\ge z}\overline{Q}_k\PAR{\rd z,\rd\widetilde \theta,\rd r},
	\end{align*}
	where the last line follows from Assumption~\ref{Hyp:Kernel} and $\fF^N(r^-)\le\lambda_*,\,\gamma\le1$, and where $\overline{Q}_k$ is the compensated Poisson measure of $Q_k$.
	\\
	Hence for $\delta$ small enough that $\left(\|\partial_af\|_\infty+2\lambda_*\|f\|_\infty\right)\delta\le\frac{\epsilon}{2}$, and by Doob's maximal inequality, we have
	\begin{align*}
		&\mathbb{P}\left(\sup_{0\le v\le \delta}\left|\langle\mu ^N_{t+v},f\rangle-\langle\mu ^N_t,f\rangle\right|\ge\epsilon\right)\\
&\le	\mathbb{P}\left(\sup_{0\le v\le \delta}	\left|\frac{1}{N}\sum_{k=1}^N\int_{t}^{t+v}\hskip-0.2cm\int_{\Theta\times\dR_+}\hskip-0.4cm\mathds{1}_{\fF^N(r^-)\gamma^N_k(r^-)K(\theta^N_k(r^-),\widetilde\theta)\ge z}\overline{Q}_k\PAR{\rd z,\rd\widetilde \theta,\rd r}\right|\ge\frac{\epsilon}{4\|f\|_\infty}\right)\\
&\le
\frac{16\|f\|_\infty^2}{\epsilon^2}\dE\left[\left(\frac{1}{N}\sum_{k=1}^N\int_{t}^{t+\delta}\hskip-0.2cm\int_{\Theta\times\dR_+}\hskip-0.3cm\mathds{1}_{\fF^N(r^-)\gamma^N_k(r^-)K(\theta^N_k(r^-),\widetilde\theta)\ge z}\overline{Q}_k\PAR{\rd z,\rd\widetilde \theta,\rd r}\right)^2\right]\\
		&=
\frac{16\|f\|_\infty^2}{N\epsilon^2}\dE\left[\left(\int_{t}^{t+\delta}\int_{\Theta\times\dR_+}\mathds{1}_{\fF^N(r^-)\gamma^N_1(r^-)K(\theta^N_1(r^-),\widetilde\theta)\ge z}\overline{Q}_1\PAR{\rd z,\rd\widetilde \theta,\rd r}\right)^2\right]\\
		&=
\frac{16\|f\|_\infty^2}{N\epsilon^2}\int_{t}^{t+\delta}\int_{\Theta}\dE\left[\fF^N(r^-)\gamma^N_1(r^-)K(\theta^N_1(r),\theta)\right]\nu(\rd\theta)\rd r\\
		&\le
\frac{16\|f\|_\infty^2\lambda_*}{N\epsilon^2}\delta,
	\end{align*}
	where the third line follows from the orthogonality since $(Q_k)_{k\ge 1}$ are independent and the exchangeability of the family of processes $(a^N_k,\theta^N_k,Q_k)_{k\ge 1}$, and the last line follows from the upper bounds of $\lambda,\gamma$ and from Assumption~\ref{Hyp:Kernel}.
We then deduce that
\begin{equation*}
	\lim_{\delta\to0}\limsup_{N}\sup_{0\le t\le T}\frac{1}{\delta}\mathbb{P}\left(\sup_{0\le v\le \delta}\left|\langle\mu ^N_{t+v},f\rangle-\langle\mu ^N_t,f\rangle\right|\ge\epsilon\right)=0,
\end{equation*}
and  $\PAR{\langle\mu ^N_t,f\rangle}_{N\ge 2}$ is tight by Lemma~\ref{VVM-Lem-20}. 
\end{proof}

Let $f$ be a bounded measurable function defined on $\dR_+\times\dR_+\times\dR^d$ with values in $\dR$. We introduce
\begin{align*}
	\cZ^Nf(t)=\frac{1}{N}\sum_{k=1}^N\int_{0}^{t}\int_{\Theta\times\dR_+}&\left(f_{s}(0,\widetilde\theta)-f_s(a^N_k(s^-),\theta^N_k(s^-))\right)\\
    &\times\mathds{1}_{\fF^N(s^-)\gamma^N_k(s^-)K(\theta^N_k(s^-),\widetilde\theta)\ge z}\overline{Q}_k\PAR{\rd z,\rd\widetilde \theta,\rd s},
\end{align*}
where $\PAR{\overline{Q}_k}_{1\le k\le N}$ are the compensated Poisson measures of $(Q_k)_{1\le k\le N}$.

\begin{lemma}\label{VVM-lem-M}
	For all $0\le t\le T$, for all bounded functions $f$, as $N\to\infty$,
	\begin{equation}
	\dE\left[\sup_{0\le t\le T}\left(\cZ^Nf(t)\right)^2\right]\to0.
	\end{equation}
\end{lemma}
\begin{proof}
	Note that $\PAR{\cZ^Nf(t)}_{t\ge 0}$ is a martingale. First by the Burkholder-Davis-Gundy inequality, and since the family $(Q_k)_{k\ge 1}$ is independent, the family $(\overline{Q}_k)_{k\ge 1}$ is orthogonal, it follows that
	\begin{align*}
		&\dE\left[\sup_{0\le t\le T}\left(\cZ^Nf(t)\right)^2\right]
		\\
		&\le\frac{4}{N^2}\sum_{k=1}^N\int_{0}^{T}\int_{\Theta}\dE\left[\left(f_{s}(0,\theta)-f_s(a^N_k(s),\theta^N_k(s))\right)^2\fF^N(s)\gamma^N_k(s)K(\theta^N_k(s),\theta)\right]\nu(\rd\theta)\rd s\\
		&=\frac{4}{N}\int_{0}^{T}\int_{\Theta}\dE\left[\left(f_{s}(0,\theta)-f_s(a^N_1(s),\theta^N_1(s))\right)^2\fF^N(s)\gamma^N_1(s)K(\theta^N_1(s),\theta)\right]\nu(\rd\theta)\rd s\\
		&\le\frac{16\|f\|_\infty^2T\lambda_*}{N},
	\end{align*}
	where the third line follows from exchangeability and the last line from the upper bound conditions on $\lambda,\gamma$ and from Assumption~\ref{Hyp:Kernel}.
	This concludes the proof. 
\end{proof}

We deduce from \eqref{VV-eq-m} that the limit of any converging subsequence of $(\mu^N)$ satisfies Equation~\eqref{VVM-eq1}, and by the uniqueness stated in Proposition~\ref{VVM-ex-th}, Theorem~\ref{VVM-th} is proved.

\section{Long time behavior}\label{sec-proofs-endemic-v}

Assume that $\mu_0$ has a density $u_0$ with respect to $\rd a \nu(\rd \theta)$ on $\dR_+\times \Theta$.
Let us recall by Proposition~\ref{AP-density} that the limit measure $\mu_t$ has a density $u_t$ solution to \eqref{VVL-eq-1}, and by the method of characteristics, it is given by \begin{equation*}\label{VVL-eq-2}
u_t(a,\theta)=\begin{cases}
u_0(a-t,\theta)\exp\left(-\int_{0}^{t}\fF(s)\gamma(a-t+s,\theta)\rd s \right)&\mbox{ if }a> t\\
\fF(t-a)\fS (t-a,\theta)\exp\left(-\int_{t-a}^{t}\fF(s)\gamma(s-t+a,\theta)\rd s \right)&\mbox{ if }a\le t,
\end{cases}
\end{equation*}
where $\fF$ and $\fS $ are respectively defined by \eqref{eq:def-F} and \eqref{eq:def-S}.

\subsection{Existence of a stationary measure}\label{sec-proofs-endemic-v-ex}

Adapting the proof of \cite{calsina2013steady} to study the equilibria of \eqref{VVM-eq1}, we prove Theorem~\ref{thm:existence-endemicity} in this section. We assume throughout this section that  Assumptions~\ref{Hyp:Kernel} and \ref{Hyp:endemicity} are satisfied. \\
If it exists, a stationary solution $u_*$ of \eqref{VVL-eq-1} is a solution the following system:
\begin{equation}\label{VVM-eq-st}
\begin{cases}
u_*(a,\theta)=\fF_*\fS _*(\theta)\exp\left(-\fF_*\int_{0}^{a}\gamma(s,\theta)\rd s\right)
\\[0.3cm]
\fF_*=\int_{\dR_+\times\Theta}\lambda(a,\theta)u_*(a,\theta)\rd a\nu(\rd\theta)\\[0.3cm]
\fS _*(\theta)=\int_{\dR_+\times\Theta}\gamma(a,\widetilde\theta)K(\widetilde\theta,\theta)u_*(a,\widetilde\theta)\rd a\nu(\rd\widetilde{\theta}).
\end{cases}
\end{equation}

Obviously $u_*\equiv 0$ is solution to \eqref{VVM-eq-st}. We are looking for a criterion for the existence of probability density solutions on $\dR_+\times\Theta$ to \eqref{VVM-eq-st}, i.e., such that
\begin{equation}\label{eq:u*-density}
\int_0^\infty\int_\Theta u_*(a,\theta)\rd a\nu(\rd \theta)=1.
\end{equation}
To achieve this, we study the existence of a non-negative solution $\left(x,\fS \right)$, denoted  by $\left(\fF_*,\fS _*\right)$, to the following system: 
    \begin{numcases}{}
        x\int_{\dR_+\times\Theta}\fS (\theta)\exp\left(-x\int_0^a\gamma(s,\theta)\rd s\right)\rd a\nu(\rd\theta)=1\label{VVM-eq-st-1}\\
x=x\int_{\dR_+\times\Theta}\lambda(a,\theta)\fS (\theta)\exp\left(-x\int_0^a\gamma(s,\theta)\rd s\right)\rd a\nu(\rd\theta)\label{VVM-eq-st-2'}\\
      \fS (\theta)=x\int_{\dR_+\times\Theta}\gamma(a,\widetilde\theta)K(\widetilde\theta,\theta)\fS (\widetilde\theta)\exp\left(-x\int_0^a\gamma(s,\widetilde\theta)\rd s\right)\rd a\nu(\rd\widetilde\theta),  \label{VVM-eq-st-2}
    \end{numcases}
    where the first equation is related to \eqref {eq:u*-density} to ensure that $u_*$ is a probability density function, and the two other equations come from the boundary conditions of \eqref{VVM-eq-st}.
    
By  Assumption~\ref{Hyp:endemicity}.\ref{Hyp:mean-susceptibility}, we first note that $\int_{0}^{\infty}\gamma(a,\cdot)\rd a=\infty$ $\nu$-a.e, and then $\forall x>0$,
\[
x\int_{\dR_+}\gamma(a,\cdot)\exp\left(-x\int_{0}^{a}\gamma(s,\cdot)\rd s\right)\rd a=1\quad \nu\text{-a.e}.
\]
Consequently, Equation~\eqref{VVM-eq-st-2} becomes 
\[
 \fS (\theta)=\int_{\Theta}K(\widetilde\theta,\theta)\fS (\widetilde\theta)\nu(\rd\widetilde\theta).
\]
    
We introduce the linear operator $T$ from $L^1(\nu)$ to $L^1(\nu)$ defined by 
\begin{equation}\label{VV-op-T}
T(B)(\theta)= \int_{\Theta}K(\widetilde\theta,\theta)B(\widetilde\theta)\nu(\rd\widetilde{\theta}).
\end{equation}
\begin{remark}\label{rk:spectreT}
We easily observe, by Assumption~\ref{Hyp:Kernel}, $\forall B\in L^1(\nu)$, 
\begin{equation*}\label{VV-op-T-a}
\NRM{T(B)}_{L^1(\nu)}\le \NRM{B}_{L^1(\nu)}\quad \text{ and }\quad
    \int_\Theta T(B)(\theta)\nu(\rd\theta)=\int_\Theta B(\theta)\nu(\rd\theta).
\end{equation*}
 It follows that the spectral radius $\rho(T)$ of $T$ is smaller than $1$ and the only possible eigenvalue with an integrable nonnegative eigenfunction is $1$.
Moreover, for eigenvalues different from $1$, the associated eigenfunctions $B$ satisfy
 $\displaystyle{\int_\Theta B(\theta)\nu(\rd\theta)=0}$.
 \end{remark}

Recalling Equation~\eqref{VVM-eq-st-2}, we are looking for an integrable nonnegative solution $\fS $ of $ \fS =T(\fS )$. 
We now mention the following spectral property of integral operators.

\begin{theorem}{\cite[Theorem~$6.6$, Chapter 5]{schaefer1974banach}}\label{thm:spectreT}
	Let $(\Theta,\cH,\nu)$ is a $\sigma$-finite measure space and $E:=L^p(\nu),$ where $1\le p\le\infty$. Let $T$ be an integral linear operator on $E$ given by a measurable kernel $K\ge0$ fulfilling the two conditions
	\begin{enumerate}
		\item some power of $T$ is compact;
		\item\label{Ac-th-3} for any $S\in\cH$ such that $\nu(S)>0$, 
  and $\nu(\Theta\setminus S)>0$ 
		\[\int_{\Theta\setminus S}\int_S K(\widetilde\theta,\theta)\nu(\rd\theta)\nu(\rd\widetilde\theta)>0.\]
	\end{enumerate}
Then the spectral radius $\rho(T)$ of $T$ is a strictly positive eigenvalue, 
and it admits a unique normalized eigenfunction $\fS $ satisfying $\fS (.)>0$ $\nu$-a.e.
Moreover, if $K>0$ $\nu\otimes\nu$-a.e., then any eigenvalue $\kappa$ of $T$ different from $\rho(T)$ has modulus $|\kappa|<\rho(T)$.    
\end{theorem}
From Remark~\ref{rk:spectreT} and Theorem \ref{thm:spectreT}, we easily deduce the following proposition.
\begin{proposition}
    \label{VVM-th-F-propre}
    Under Assumptions~\ref{Hyp:Kernel} and~\ref{Hyp:endemicity}, the spectral radius of the operator $T$ defined by \eqref{VV-op-T} is $\rho(T)=1$. There is a unique eigenfunction $\fS _*$, positive $\nu$-a.e and in $L^1(\nu)$, associated to the eigenvalue $1$, such that
\begin{equation}\label{eq:normalization-eigenfunction}
\int_{\dR_+\times\Theta}\lambda(a,\theta)\fS _*(\theta)\rd a\nu(\rd\theta)=1.
  \end{equation}
\end{proposition}

\begin{proof}
From Assumption~\ref{Hyp:endemicity}.\ref{Hyp:kernel in L^1} , we have 
$
\int_{\Theta}\sup_{\widetilde\theta\in\Theta}K(\widetilde\theta,\theta)\nu(\rd\theta)<\infty ,
$
consequently the operator $T$ is of Hille-Tamarkin type (see \cite[Section 11.3]{jorgens1982linear}) and it follows from \cite[Theorem~$11.9$]{jorgens1982linear} that the square of the operator $T^2$ from $L^1(\nu)$ to $L^1(\nu)$ is compact. 
Moreover, as the application $\theta\mapsto K(.,\theta)$ is a positive density on $\Theta$, Condition~(\ref{Ac-th-3})  of Theorem~\ref{thm:spectreT} is also satisfied.
We then deduce that the spectral radius $\rho(T)$ of the operator $T$ is a strictly positive, an isolated simple eigenvalue of $T$, and it is the only eigenvalue with a corresponding normalized positive $\nu$-a.e. eigenfunction.
\\
Moreover, by Remark~\ref{rk:spectreT}, we note that $\rho(T)=1$. We then define $\fS _*$ as the unique positive eigenfunction  of $T$  in $L^1(\nu)$ associated with the eigenvalue $1$ such that Condition~\eqref{eq:normalization-eigenfunction} is satisfied. We note that $\fS _{*}$ satisfying \eqref{eq:normalization-eigenfunction} is well defined since $0\le\lambda\le\lambda_*$.
\end{proof}

Let $\fS _*$ be given by Proposition~\ref{VVM-th-F-propre}. We are now looking for $x>0$ such that \eqref{VVM-eq-st-1} holds. 
To this end, we introduce the function $H:\dR_+\to\mathbb{R}_+$ defined by
\begin{equation}\label{VVM-eq-H}
H(x)=x\int_{\dR_+\times\Theta}\exp\left(-x\int_0^a\gamma(s,\theta)\rd s\right)\fS _*(\theta)\rd a\nu(\rd\theta).
\end{equation}
Our goal is to find a solution to $H(x)=1$ on $(0,+\infty)$.
\begin{lemma}\label{VVM-lem-ex}
	Under Assumption~\ref{Hyp:endemicity}, the function $H:\dR_+\to\dR_+$ is well defined and continuous, and
	\[	H(0)=\int_{\Theta}\frac{1}{\gamma_*(\theta)}\fS _*(\theta)\nu(\rd \theta)\text{ and }\lim_{x \to +\infty}H(x)=+\infty.\]
     When Condition~\eqref{eq:condition-endemic} is satisfied, i.e. when 
    $\int_{\Theta}\frac{1}{\gamma_*(\theta)}\fS _*(\theta)\nu(\rd \theta)<1$,
    there exists $\fF_*>0$ such that $H(\fF_*)=1.$
\end{lemma}
\begin{proof}
	Since $\fS _*\in L^1(\nu)$, from Assumption~\ref{Hyp:endemicity}.\ref{Hyp:ess sup}, the function $H$ is well defined and continuous. 
	Using successive changes of variables $b=ax$ and $u=xs$, we have
	\begin{align}
	H(x)&=\int_{\dR_+\times\Theta}\exp\left(-x\int_0^{b/x}\gamma(s,\theta)\rd s\right)\fS _*(\theta)\rd b\nu(\rd\theta)\label{eq:H1}\\
    &=\int_{\dR_+\times\Theta}\exp\left(-\int_0^{b}\gamma\left(\frac{u}{x},\theta\right)\rd u\right)\fS _*(\theta)\rd b\nu(\rd\theta).\label{eq:H}
	\end{align}
	Hence, using Assumption~\ref{Hyp:endemicity}.\ref{Hyp:mean-susceptibility} for the limit when $x$ goes to $0$, and using  $\gamma\in[0,1]$ with $\gamma(0,.)\equiv 0$ and Fatou's Lemma for the limit when $x$ goes to $+\infty$, it follows that,
	\begin{equation*}
	H(0)=\int_{\Theta}\frac{1}{\gamma_*(\theta)}\fS _*(\theta)\nu(\rd \theta),\text{ and }\lim_{x \to +\infty}H(x)=+\infty.
	\end{equation*} 
The conclusion follows.
\end{proof}

We can now give the proof of Theorem~\ref{thm:existence-endemicity}.
\begin{proof}[Proof of Theorem~\ref{thm:existence-endemicity}]
We first assume that $\displaystyle{\int_{\Theta}\frac{1}{\gamma_*(\theta)}\fS _*(\theta)\nu(\rd \theta)<1}$.
    Let  $\fF_*$ be given by Lemma~\ref{VVM-lem-ex},
we consider the function $u_*$ defined on $\dR_+\times\Theta$ by
\[
u_*(a,\theta)=\fF_*\fS _*(\theta)\exp\left(-\fF_*\int_{0}^{a}\gamma(s,\theta)\rd s\right).
\]
By Assumption~\ref{Hyp:endemicity}.\ref{Hyp:disjoint-support},  $\mathrm{Supp}(\lambda)\cap \mathrm{Supp}(\gamma)= \emptyset$, we easily observe that $\forall x\ge0$,
  \begin{align}\label{eq:int-sigma*}
  \int_{\dR_+\times\Theta}\hskip-0.2cm\lambda(a,\theta)\fS _*(\theta)\exp\left(-x\int_{0}^{a}\gamma(s,\theta)\rd s \right)\rd a\nu(\rd\theta)
=\int_{\dR_+\times\Theta}\hskip-0.2cm\lambda(a,\theta)\fS _*(\theta)\rd a\nu(\rd\theta)=1.
  \end{align}
Since the couple $(\fF_*,\fS _*)$ satisfies the system of equation \eqref{VVM-eq-st-1}-\eqref{VVM-eq-st-2'}, and by \eqref{eq:int-sigma*},
we deduce that $u_*$ is a solution to the system \eqref{VVM-eq-st}.
\\
Moreover, using Expression~\eqref{eq:H1} of  $H$ to compute its first derivative, we have
\begin{align*}
	H'(x)&=-\int_{\dR_+\times\Theta}\PAR{\int_0^{b/x}\gamma(s,\theta)\rd s+\frac{b}{x}\gamma\PAR{\frac{b}{x},\theta}}\re^{-x\int_0^{b/x}\gamma(s,\theta)\rd s}\fS _*(\theta)\rd b\nu(\rd\theta)\\
 &=x\int_{\dR_+\times\Theta}\PAR{-\int_0^{a}\gamma(s,\theta)\rd s+a\gamma\PAR{a,\theta}}\re^{-x\int_0^{a}\gamma(s,\theta)\rd s}\fS _*(\theta)\rd a\nu\rd\theta).
 \end{align*}
 Consequently, when Condition~\eqref{eq:condi-uniq-endemic} is satisfied,
 $H$ is a non-decreasing function and the conclusion follows.

 To conclude the proof of Theorem~\ref{thm:existence-endemicity}, we introduce the infection duration $T_I$: for $a>T_I(\theta)$, $\lambda(a,\theta)=0$. The proportion of infected individuals at the endemic equilibrium is thus given by 
\begin{align*}
I_*&=\int_{\dR_+\times\Theta} \ind_{T_I(\theta)>a}u_*(a,\theta)\rd a\nu(\rd\theta) \\
&=\fF_*\int_{\dR_+\times\Theta}\indic{T_I(\theta)>a}\fS _*(\theta)\exp\left(-\fF_*\int_{0}^{a}\gamma(s,\theta) d s\right)\rd a\nu(\rd\theta)\\
&=\fF_*\int_{\Theta} T_I(\theta)\fS _*(\theta)\nu(\rd\theta),
\end{align*}
since $ \mathrm{Supp}(\gamma(\cdot, \theta))\subset\PAR{T_I(\theta),\infty}$ by Assumption~\ref{Hyp:endemicity}.\ref{Hyp:disjoint-support}.
\end{proof}

\begin{remark}
    When $a\mapsto\gamma(a,.)$ is $\nu$-a.e. non-decreasing, Condition~\eqref{eq:condi-uniq-endemic} is satisfied, and $H$ given by \eqref{eq:H} is obviously non-decreasing.
But as we will notice in Section~\ref{sec:vaccin-renewal} where we introduce vaccination policies, Condition~\eqref{eq:condi-uniq-endemic} is not a necessary condition for $H$ to be non-decreasing.
    
  Note that when $H$ is non-decreasing, then Condition \eqref{eq:condition-endemic} becomes a sufficient and necessary condition for the existence of an endemic equilibrium.
\end{remark}

Even if Condition~\eqref{eq:condi-uniq-endemic} is not a necessary condition for $H$ to be non-decreasing, we remark in the following example that it can be an optimal condition for the monotony of $H$.
\begin{example}\label{exple:monotony-H} Let $(\Theta,\cH,\nu)$ be a probability space.
 Fix $\alpha, \beta\in(0,1]$, with $\alpha\ge \beta$, and $T_V,T_R$ real values such that $0<T_R<T_V$. We consider  the non-monotone function $\gamma(a)=\alpha\ind_{T_R<a<T_V}+\beta\ind_{a\ge T_V}$,  independent of $\theta$. Assumption~\ref{Hyp:endemicity}.\ref{Hyp:mean-susceptibility} is satisfied with $\gamma_*=\beta$.
We easily compute from \eqref{VVM-eq-H}  that
\[
H(x)=\kappa\PAR{xT_R+\frac{1}{\alpha}-\PAR{\frac{1}{\alpha}-\frac{1}{\beta}}\re^{-x(T_V-T_R)\alpha}},
\]
where $\kappa=\int_\Theta \fS _*(\theta)\nu(\rd\theta)$. We observe that $H$ is non-decreasing if and only if $\beta T_V\ge \alpha(T_V-T_R)$.

On the other hand, we have
\begin{align*}
\int_0^a\gamma(s)\rd s=\alpha\PAR{a-T_R}\ind_{T_R<a<T_V}+\PAR{\alpha(T_V-T_R)+\beta(a-T_V)}\ind_{a\ge T_V},
\end{align*}
which satisfies $\gamma(a)\ge \frac{1}{a}\int_0^a\gamma(s)\rd s$ for any $a\ge 0$ if and only if $\beta T_V\ge \alpha(T_V-T_R)$.
Note that the long time behaviour of a generalization of this model is detailed in Section~\ref{sec:vaccin-one-shot}. 
\end{example}

We now compute the value of $\fS _*$ for specific kernels $K(\cdot,\cdot)$.
\begin{example}\label{ex:values of S_*}
   \begin{enumerate}
   \item \label{ex:no memory} When there is no memory of the previous infections, we have $K(\widetilde\theta):=K(\theta,\widetilde\theta)$ independent of $\theta$ with $\int_{\Theta}K(\widetilde\theta)\nu(\rd\widetilde\theta)=1$. Up to a change of measure on $\PAR{\Theta, \cH, \nu}$, we can assume that $K\equiv 1$ and $\nu$ is a probability measure. Then, we easily deduce that $\fS _*$ is a constant, and using Condition~\eqref{eq:normalization-eigenfunction}, we have
   \[
   \fS _*=\PAR{\int_{\dR_+\times\Theta}\lambda(a,\theta)\rd a\nu(\rd\theta)}^{-1}=\frac{1}{R_0},
   \]
   where $R_0$ is the basic reproduction number. 
Let us recall that under \cite[Assumption 4.1 and 4.2]{forien-Zotsa2022stochastic}, Assumption~\ref{Hyp:endemicity} is satisfied with $\gamma_*(\theta)=\lim_{a\to \infty}\gamma(a,\theta)$. 
Then, by Theorem~\ref{thm:existence-endemicity}, there is existence of an endemic equilibrium when
\[
R_0>\int_\Theta\frac{1}{\gamma_*(\theta)}\nu(\rd\theta)=\dE_\nu\SBRA{\frac{1}{\gamma_*}},
\] 
 and, in this case, the proportion of infected individuals at the endemic equilibruim is \[I_*=\frac{\fF_*\dE_\nu[T_I]}{R_0},\]
with $\dE_\nu$ the expectation on $\Theta$ w.r.t. the probability measure $\nu$.
We  recover the result of  \cite[Theorem~$4.2$]{forien-Zotsa2022stochastic}.

    \item When $K(\cdot,\cdot)$ is symmetric, for instance when it depends on the distance between the parameters, so that for each $\theta,\widetilde\theta\in \Theta,\,K(\widetilde{\theta},\theta)=K(\theta,\widetilde\theta)$, and when $\nu$ is a probability measure (so that Condition~\ref{Hyp:endemicity}.\ref{Hyp:kernel in L^1} may be satisfied), Assumption~\ref{Hyp:Kernel} implies that
    \[\int_\Theta K(\widetilde{\theta},\theta)\nu(\rd \widetilde{\theta})=\int_\Theta K(\theta,\widetilde{\theta})\nu(\rd \widetilde{\theta})=1.\]
    Consequently, by Proposition~\ref{VVM-th-F-propre}, $\fS _*=\frac{1}{R_0}$ and the condition of existence of an endemic equilibrium is
    \[
    R_0>\int_\Theta\frac{1}{\gamma_*(\theta)}\nu(\rd\theta).    \]
    \item We assume that there exists a density $\pi$ on $\Theta$ such that, 
\[\pi(\theta)=\int_{\Theta}K(\widetilde{\theta},\theta)\pi(\widetilde{\theta})\nu(\rd\widetilde{\theta}).\]
In particular, reversible kernels with respect to $\pi$ satisfies this property. By Proposition~\ref{VVM-th-F-propre}, $\fS_*$ is equal to $\pi$ up to a constant. Since $\fS_*$ has been chosen such that $\int_{\dR_+\times\Theta}\lambda(a,\theta)\fS _*(\theta)\rd a\nu(\rd\theta)=1$,
we obtain
\[\fS _*(.)=\pi(.)\left(\int_{\dR_+\times\Theta}\lambda(a,\theta)\pi(\theta)\rd a\nu(\rd\theta)\right)^{-1}.\]
Therefore, there is an endemic equilibrium if 
$
\dE_\nu^*\SBRA{\frac{1}{\gamma_*}}<R_0^*,
$
 with $R_0^*=\dE_\nu^*\SBRA{\int_0^\infty\lambda(a)\rd a}$ and $\dE_\nu^*$ the expectation  on $\Theta$ w.r.t. the measure $\pi(\theta)\nu(\rd\theta)$. 

 Note that, kernels with the same property appear in the literature on mathematical ecology (see e.g.~\cite{cantrell2012evolutionary} which considers such free dispersal strategy kernels).
\item We take $\Theta=\BRA{\theta_1,\theta_2}$, with
$\nu=p_1\delta_{\theta_1}+p_2\delta_{\theta_2}$, $p_2=1-p_1\in(0,1)$, and 
$K=\PAR{K_{ij}}_{1\le i,j\le 2}$, where $K_{ij}=K(\theta_i,\theta_j)$ with $K_{11}p_1+K_{12}p_2=K_{21}p_1+K_{22}p_2=1$ (Assumption~\ref{Hyp:Kernel} is thus satisfied). From \eqref{VV-op-T}, we consider the matrix
\[
T=\begin{pmatrix}
    K_{11}p_1&K_{21}p_2\\
    K_{12}p_1&K_{22}p_2
\end{pmatrix}.
\]
We notice that $\mathrm{Span}\BRA{\begin{pmatrix}
    K_{21}\\K_{12}
\end{pmatrix}}$ is the eigenspace associated to the eigenvalue $1$ and $\mathrm{Span}\BRA{\begin{pmatrix}
   -p_2\\p_1
\end{pmatrix}}$ is the eigenspace associated to the eigenvalue $p_2\PAR{K_{22}-K_{12}}$. By definition,
$\fS _*$ is the eigenfunction associated with the eigenvalue $1$ such that $\displaystyle{\int_{\dR_+\times \Theta} \lambda(a,\theta)\fS _*(\theta)\rd a\nu(\rd\theta)=1}$. We deduce
\[
\begin{pmatrix}
\fS _*(\theta_1)\\
\fS _*(\theta_2)
\end{pmatrix}
=\frac{1}{p_1K_{21}\int_0^\infty\lambda(a,\theta_1)\rd a+p_2K_{12}\int_0^\infty\lambda(a,\theta_2)\rd a}\, \begin{pmatrix}
    K_{21}\\K_{12}
\end{pmatrix}.
\]
The condition of existence of an endemic equilibrium is then
\[
\dE_\nu^*\SBRA{\frac{1}{\gamma_*}}<R_0^*\quad \text{with }R_0^*=\dE_\nu^*\SBRA{\int_0^\infty\lambda(a)\rd a}
\]
where $\dE_\nu^*$ is the expectation with respect to the measure \\$\PAR{\frac{p_1K_{21}}{p_1K_{21}+p_2K_{12}},\frac{p_2K_{12}}{p_1K_{21}+p_2K_{12}}}$ on $\Theta=\BRA{\theta_1,\theta_2}$.
\\
As noted in Remark~\ref{rk:spectreT}, we observe that all eigenvectors $B$ associated with the eigenvalue $p_2\PAR{K_{22}-K_{12}}$ satisfy $\displaystyle{\int_{\Theta} B(\theta)\nu(\rd \theta)=B_1p_1+B_2p_2=0}$.
\end{enumerate}
\end{example}

\subsection{About the local stability of endemic equilibria}\label{sec:local stability}

To deal with  the asymptotic stability of the equilibrium, we use the tools of abstract semi-linear Cauchy problems. We refer the reader to 
\cite{magal2018theory,thieme1990semiflows,webb1985theory} for more details.

We assume throughout this section that Assumptions~\ref{Hyp:Kernel}, \ref{Hyp:endemicity}, and \ref{Hyp:endemic-stability} are satisfied, as well as Condition~\eqref{eq:condition-endemic}.
 We also assume that $\Theta$ is an open subset of $\dR^d$ and $\nu$ is a  $\sigma$-finite measure absolutely continuous with respect to the Lebesgue measure with support on $\Theta$. 
In this section, we keep the memory of the last infection in the proof as far as possible, but will have to remove it in the last step to obtain semi-explicit conditions on the infectivity and the susceptibility curves ensuring local stability in a general setting. 

\medskip
	We set $\dX=L^1(\nu)\times L^1(\rd a\otimes\nu)$ and endow $\dX$ with the product norm \begin{equation}\label{eq:norm-X}
    \|(\psi,\phi)\|_{\dX}=\|\psi\|_{L^1(\nu)}+\|\phi\|_{L^1(\rd a\otimes\nu)}.
    \end{equation}

	We introduce $\dX_0=\{0\}\times L^1(\rd a\otimes\nu)$ and
    the Sobolev space $W^{1,1}(\dR_+\times\dR^d):=\{\varphi\in L^{1}(\rd a\otimes\nu) \text{ such that }\partial_a\varphi\text{ exits and }\partial_a\varphi\in L^{1}(\rd a\otimes\nu)\}$. 
    We set $D(\mathcal{A})=\{0\}\times W^{1,1}(\dR_+\times\dR^d)$. Note that $\overline{D(\mathcal{A})}=\dX_0$ (see \cite[Chapter 8, p. 354]{magal2018theory} for a similar construction). 
		 
	We define, for $\left(0,\phi\right)\in D(\mathcal{A})$, the operator
	\begin{equation}
		\mathcal{A}\left(\begin{aligned}&0\\&\phi\end{aligned}\right)(a,\theta)=\left(\begin{aligned}
		&-\phi(0,\theta)\\&-\partial_a \phi(a,\theta)\end{aligned}\right) ,
		\end{equation}
		and for $\phi\in L^1(\rd a\otimes\nu)$, the operator
		\begin{equation*}
		F(\phi)=F\left(\begin{aligned}0\\\phi\end{aligned}\right)=\left(\begin{aligned}
		F_0(\phi)\\F_1(\phi)\end{aligned}\right)\quad\text{with}\quad \begin{array}{l}
F_0(\phi)(\theta)=\SCA{\lambda,\phi}\SCA{ \gamma K(\cdot,\theta),\phi},\\[0.2cm]
F_1(\phi)(a,\theta)=-\SCA{\lambda,\phi}\gamma(a,\theta)\phi(a,\theta),
\end{array}
		\end{equation*}
where $\SCA{\lambda,\phi}$ and  $\SCA{ \gamma K(\cdot,\theta),\phi}$ denote
\begin{align*}
\SCA{\lambda,\phi}&=\int_{\dR_+\times\Theta}\lambda( a,\theta)\phi( a,\theta)\rd a\nu(\rd\theta)\\
\SCA{ \gamma K(\cdot,\theta),\phi}&=\int_{\dR_+\times\Theta}\gamma( a,\widetilde\theta)K(\widetilde\theta,\theta)\phi(a,\widetilde\theta)\rd a\nu(\rd\widetilde{\theta}).
\end{align*}
By Assumption~\ref{Hyp:Kernel}, we remark that
\[
\int_\Theta\SCA{\gamma K(\cdot,\theta),\phi}\nu(\rd\theta)=\int_{\dR_+\times\Theta}\gamma( a,\theta)\phi(a,\theta)\rd a\nu(\rd\theta).
\]

	We note that	 $\mathcal{A}$ is the infinitesimal generator of the following strongly continuous semigroup on $\dX_0$ (see, e.g. \cite[Theorem 1.3.1]{magal2018theory}): 
		 \begin{equation}
		 	T_\mathcal{A}(t)\left(\begin{aligned}
		 	&0\\&\phi
		 	\end{aligned}\right)(a,\theta)=\left(\begin{aligned}
		 	&\quad\quad0\\&\phi(a-t,\theta)\mathds{1}_{a\ge t}
		 	\end{aligned}\right).
		 	\end{equation}
			
	In the spirit of Thieme~\cite{thieme1990semiflows}, we thus can rewrite the PDE~\eqref{VVL-eq-1} as follows:
	\begin{equation} \label{VVM-eq-weak-r}
		\left\{
		\begin{aligned}
		&\partial_t v_t=\mathcal{A}(v_t)+F(v_t)\\
		&v(0,a,\theta)=v_0(a,\theta)
		\end{aligned}\right.
	\end{equation}
	where \[v_t(a,\theta)=\left(\begin{aligned}&\quad 0\\&u_t(a,\theta)\end{aligned}\right)\quad\text{ and } \quad v_0(a,\theta)=\left(\begin{aligned}&\quad 0\\&u_0(a,\theta)\end{aligned}\right),\]
  with for each $t\ge0$, $\int_{\dR_+\times\Theta} u_t(a,\theta)\rd a \nu(\rd\theta)=1$.
    Hence the operator $\mathcal{A}$ is studied on the space 
    \begin{equation}
    \BRA{v=(0,u)\in\mathbb{X}_0:\int_{\dR_+\times\Theta} u(a,\theta)\rd a\nu(\rd\theta)=1}.\label{eq-dom-A}\end{equation}
With this new formulation, the boundary condition of the  PDE~\eqref{VVL-eq-1} has been integrated in the perturbation $F$ of Equation~\eqref{VVM-eq-weak-r}.

Since the assumptions of Theorem~\ref{thm:existence-endemicity} and \eqref{eq:condition-endemic}  are satisfied, we introduce $u_*$ an endemic equilibrium, defined as a probability density on $\dR_+\times \Theta$ solution to the system \eqref{VVM-eq-st}.
We note that any equilibrium $u_*$ satisfies
$
\cA(u_*)+F(u_*)=0$.

In what follows, the arguments are inspired by Thieme's~\cite{thieme1990semiflows} and Webb's \cite{webb1985theory}.
We note that Equation~\eqref{VVM-eq-weak-r} is non-linear, due to the non-linearity of $v\mapsto F(v)$. As $F$ is Frechet-differentiable, with derivative $F'(u_*)$, at the equilibrium point $(0,u_*)$, following   \cite{webb1985theory, magal2018theory}, to study the stability of the endemic equilibrium $u_*$, we first linearize Equation~\eqref{VVM-eq-weak-r}, replacing the non-linear part $F$ in Equation~\eqref{VVM-eq-weak-r} with its Frechet-derivative $F'(u_*)$ at equilibrium. Then we study the semigroup $\cA+F'(u_*)$, and we conclude with \cite[Theorem~$4.2$]{thieme1990semiflows}. Since Equation~\eqref{VVM-eq-weak-r}  has been linearized around the equilibrium $u_*$,  we note from \eqref{eq-dom-A} that the semigroup $\cA+F'(u_*)$ is studied on the space
    \begin{equation}\{v=(0,u)\in\mathbb{X}_0,\,\int_{\dR_+\times\Theta} u(a,\theta)\rd a\nu(\rd\theta)=0\}.\label{eq-dom-A'}\end{equation}
In our case, the Frechet-derivatives of $F_0$ and $F_1$ are given by: 
	\begin{align*}
	&F_0'(h)(\phi)(\theta)=\SCA{\lambda,h}\SCA{ \gamma K(\cdot,\theta),\phi}+\SCA{\lambda,\phi}\SCA{ \gamma K(\cdot,\theta), h}
\\
	&F_1'(h)(\phi)(a,\theta)=-\SCA{\lambda,h}\gamma(a,\theta)\phi(a,\theta)-\SCA{\lambda,\phi}\gamma(a,\theta)h(a,\theta).
	\end{align*}
We denote by $\mathcal{K}(\dR_+\times\Theta)$ the set of compact operator on $\dR_+\times\Theta$ and by $\mathcal{L}(\dR_+\times\Theta)$ the set of bounded operators on $ L^1(\rd a\otimes\nu)$.
As in Thieme~\cite[Section 4.]{thieme1990semiflows} and  Webb \cite[Proposition~$4.12$, p.~$169$]{webb1985theory}, to study the long-term behaviour of the operator $\cA+F'(u_*)$, we introduce the growth bound $w_0(\cA+F'(u_*))$ and the essential growth bound (called $\alpha$-growth bound in Webb), respectively defined by
\begin{align}\label{eq:def-w0}
    w_0(\cA+F'(u_*))&:=
    \lim_{t\to \infty}\frac{1}{t}\log\NRM{T_{\cA+F'(u_*)}(t)}_{\rm op},\\
     \label{eq:def-wess}
     w_{\rm ess}(\cA+F'(u_*))&:=
     \lim_{t\to \infty}\frac{1}{t}\log\NRM{T_{\cA+F'(u_*)}(t)}_{\rm ess},
\end{align}
 where $T_{\cA+F'(u_*)}$ is the semi-group related to the operator $\cA+F'(u_*)$, $\NRM{.}_{\rm op}$ is the operator norm defined for a semi-group $T$ by 
 \[
 \NRM{T}_{\rm op}:=\sup_{\phi:\NRM{\phi}_{ L^1(\rd a\otimes\nu)}=1}\NRM{T(\phi)}_{ L^1(\rd a\otimes\nu)},
 \] and $\NRM{.}_{\rm ess}$ is the operator norm on the quotient space $\mathcal{L}(\dR_+\times\Theta)\big/\mathcal{K}(\dR_+\times\Theta)$. 
 There is the following relation between $w_0$, $w_{\rm ess}$, and the spectrum of the operator $\cA+F'(u_*)$ (see Equation (4.57) in \cite[Proposition 4.13]{webb1985theory}):
    \begin{equation}\label{eq:relation w_0}
    w_0(\cA+F'(u_*))=\max\PAR{w_{\rm ess}(\cA+F'(u_*)),\sup_{\alpha\in{\rm sp}(\cA+F'(u_*))\setminus {\rm e_{sp}}(\cA+F'(u_*))}\cR e(\alpha)},
    \end{equation}
where ${\rm sp}(\cA+F'(u_*))$ is the spectrum of the operator $\cA+F'(u_*)$, ${\rm e_{sp}}(\cA+F'(u_*))$ is its essential spectral radius (see \cite[Definition 4.13 p.~165]{webb1985theory}), and $\cR e(\alpha)$ the real part of the complex $\alpha$. Our goal is to prove that $w_0(\cA+F'(u_*))$ is negative. We first study $w_{\rm ess}(\cA+F'(u_*))$, and in a second time we will study the spectrum of the operator.

We note that since $u_*$ is an endemic equilibrium of the PDE, we have $\SCA{\lambda,u_*}=\fF_*$ and $\SCA{ \gamma K(\cdot,\theta), u_*}=\fS _*(\theta)$ by \eqref{VVM-eq-st}.
To control $w_{\rm ess}(\cA+F'(u_*))$, we  decompose the operator as follows
\begin{align}
\left(\mathcal{A}+F'(u_*)\right)\begin{pmatrix}0\\\phi\end{pmatrix}(a,\theta)
&=\begin{pmatrix}
-\phi(0,\theta)+\fS_*(\theta)\langle\lambda,\phi\rangle+\SCA{\gamma K(\cdot,\theta), \phi}\fF_*\\-\partial_a\phi(a,\theta)-\gamma(a,\theta)\phi(a,\theta)\fF_*-\gamma(a,\theta)u_*(a,\theta)\langle\lambda,\phi\rangle\end{pmatrix}
\notag\\
&\hskip-1.5cm =\begin{pmatrix}-\phi(0,\theta)\\-\partial_a\phi(a,\theta)-\gamma(a,\theta)\phi(a,\theta)\fF_*\end{pmatrix}+
\begin{pmatrix}
\fS_*(\theta)\langle\lambda,\phi\rangle+\SCA{\gamma K(\cdot,\theta), \phi}\fF_*\\-\gamma(a,\theta)u_*(a,\theta)\langle\lambda,\phi\rangle
\end{pmatrix}
\notag\\	&\hskip-1.5cm=:\mathcal{A}_*\left(\begin{aligned}&0\\&\phi\end{aligned}\right)(a,\theta)+\mathcal{B}_*\left(\begin{aligned}&0\\&\phi\end{aligned}\right)(a,\theta).
\label{eq:def_A*-B*}
\end{align}
Lemma~\ref{lem:B*_compact} in Appendix~\ref{A:compact-operators} states that $\mathcal{B}_*$ 
is a compact operator for the norm $\NRM{.}_{\dX}$ (see \eqref{eq:norm-X}).
 Consequently, by \cite[Proposition~$4.14$, page~$179$]{webb1985theory}, 
\begin{equation}\label{eq:wess-A*}
    w_{\rm ess}(\cA+F'(u_*))=w_{\rm ess}(\cA_*).
\end{equation}
By definition $w_{\rm ess}(\cA_*)=
    \lim_{t\to \infty}\frac{1}{t}\log\NRM{T_*(t)}_{\rm ess}$, where the semi-group  $T_{\cA_*}$ generated by the operator $\cA_*$ is given thanks to the method of characteristics by
\[
T_{\cA_*}(t)\begin{pmatrix}
    0\\
    \phi
\end{pmatrix}=\begin{pmatrix}
    0\\
    T_*(t)(\phi)
\end{pmatrix}\quad \text{with}
\]
\vskip -0.4cm
\begin{align}
T_*(t)\left(\phi\right)(a,\theta)&=\mathds{1}_{a\ge t}\left[\phi(a-t,\theta)\exp\left(-\fF_*\int_{0}^{t}\gamma(a-t+r,\theta)\rd r\right)\right.\nonumber\\
&\hskip-2cm  \left.-\int_{0}^{t}\gamma(a-t+s ,\theta)u_*(a-t+s,\theta)\langle \lambda,T_*(s)(\phi)\rangle\exp\left(-\fF_*\int_{s}^{t}\gamma(a-t+r,\theta)dr\right)\rd s\right]\nonumber\\
&=:T_*^1(t)\left(\phi\right)(a,\theta)-T_*^2(t)\left(\phi\right)(a,\theta).\label{eq:T_* for a>t-ess}
\end{align}

 We remark that under Assumption~\ref{Hyp:endemic-stability}, the function $\gamma_*$ defined in Assumption~\ref{Hyp:endemicity}.\ref{Hyp:mean-susceptibility} satisfies $\gamma_*(\cdot)\ge \sigma$.
Using the compactness of the operator $T_*^2$, stated in Lemma~\ref{eq-compact-T-2} in Appendix~\ref{A:compact-operators}, we obtain the following negative upper bound for $w_{\rm ess}$.
\begin{lemma}\label{lem:bound on w_ess}
    Under Assumptions~\ref{Hyp:Kernel}, \ref{Hyp:endemicity} and \ref{Hyp:endemic-stability}, we have
\[
w_{\rm ess}(\cA+F'(u_*))\le -\fF_*\sigma.
\]
\end{lemma}
\begin{proof}
We have already noticed in \eqref{eq:wess-A*} that
	\[w_{\rm ess}(\mathcal{A}+F'(u_*))= w_{\rm ess}(\mathcal{A}_*).\]
As $T_*^2$ is a compact operator by Lemma~\ref{eq-compact-T-2}, and using \cite[Proposition~$4.9$]{webb1985theory}, we have
\begin{equation}\label{eq-nT}
 \NRM{T_*(t)}_{\rm ess}= \NRM{T_*^1(t)-T_*^2(t)}_{\rm ess}= \NRM{T_*^1(t)}_{\rm ess}\le\NRM{T_*^1(t)}_{op}.
\end{equation}
By a simple change of variables in the integrals, and using Assumption~\ref{Hyp:endemic-stability} in the inequalities, we observe that  for $\phi\in  L^1(\rd a\otimes\nu)$
    \begin{align}\label{eq:NormI2-1-ess}
    \nonumber
&\NRM{T_*^1(t)(\phi)}_{ L^1(\rd a\otimes\nu)}=\int_\Theta\int_t^\infty \ABS{\phi(a-t,\theta)}\re^{-\fF_*\int_0^t\gamma(a-t+r,\theta)\rd r}\rd a\nu(\rd \theta)\\
\nonumber
&=\int_\Theta\int_0^\infty \ABS{\phi(a,\theta)}\re^{-\fF_*\int_a^{a+t}\gamma(r,\theta)\rd r}\rd a\nu(\rd \theta)\\
\nonumber
&=\int_\Theta\int_0^{a_*} \ABS{\phi(a,\theta)}\re^{-\fF_*\int_a^{a+t}\gamma(r,\theta)\rd r}\rd a\nu(\rd \theta) \\\notag
&\hskip 3cm+\int_\Theta\int_{a_*}^\infty \ABS{\phi(a,\theta)}\re^{-\fF_*\int_a^{a+t}\gamma(r,\theta)\rd r}\rd a\nu(\rd \theta)
\\
\nonumber
&\le 
\int_\Theta\int_0^{a_*} \ABS{\phi(a,\theta)}\re^{-\fF_*\int_{a}^{a+t}\gamma(r,\theta)\rd r}\rd a\nu(\rd \theta)+\re^{-\fF_*\sigma t}\int_\Theta\int_{a_*}^\infty \ABS{\phi(a,\theta)}\rd a\nu(\rd \theta)
\\
\nonumber
&\le\re^{-\fF_*\sigma t}\left(\ind_{t\ge a_*}\re^{\fF_*\sigma a_*}\int_\Theta\int_0^{a_*} \ABS{\phi(a,\theta)}\rd a\nu(\rd \theta)\right.\\
&\hskip 2.5cm \left.+\ind_{t\le a_*}\re^{\fF_*\sigma t}\int_\Theta\int_0^{a_*} \ABS{\phi(a,\theta)}\rd a\nu(\rd \theta)
+\int_\Theta\int_{a_*}^\infty \ABS{\phi(a,\theta)}\rd a\nu(\rd \theta)\right)\nonumber
\\
&
\le \re^{-\fF_*\sigma t}\re^{\fF_*\sigma a_*}\NRM{\phi}_{L^1\PAR{\rd a\otimes \rd\nu}}.
    \end{align}

We deduce
\begin{align*}
\frac{1}{t}\log\NRM{T_*^1(t)}_{\rm op}\le -\sigma\fF_*+\frac{\fF_*\sigma a_*}{t}.
\end{align*}
We conclude using definition~\eqref{eq:def-wess} of $w_{\rm ess}(\cA_*)$ and Equation~\eqref{eq-nT}.
\end{proof}
By \eqref{eq:relation w_0}, we now need to control the real part of the eigenvalues of the operator $\mathcal{A}+F'(u_*)$. Let us take $\alpha\in\mathbb C$, and consider the following eigenvalue problem 
\begin{equation}\label{eq:eigenfunction of A+F'}
\left(\mathcal{A}+F'(u_*)\right)\left(\begin{aligned}&0\\&\phi\end{aligned}\right)(a,\theta)=\alpha\begin{pmatrix}0\\\phi(a,\theta)\end{pmatrix},
\end{equation}
where $\phi$ is an eigenfunction associated with the eigenvalue $\alpha$. We easily rewrite \eqref{eq:eigenfunction of A+F'} in the following way
\begin{numcases}{}
-\partial_a\phi(a,\theta)-\gamma(a,\theta)\phi(a,\theta)\fF_*-\gamma(a,\theta)u_*(a,\theta)\langle\lambda,\phi\rangle=\alpha\phi(a,\theta)\notag\\
\phi(0,\theta)=\fS_*(\theta)\langle\lambda,\phi\rangle+\langle \gamma K (\cdot,\theta),\phi\rangle\fF_*.\label{VVM-eq-c-1}
\end{numcases}
Consequently, an eigenfunction $\phi$ associated with the eigenvalue $\alpha$ satisfies
\begin{align}
&\phi(a,\theta)=\phi(0,\theta)\exp\left(-\int_0^a\left(\fF_*\gamma(s,\theta)+\alpha\right)\rd s\right)\nonumber\\&\hspace{3cm}-\langle\lambda,\phi\rangle\int_{0}^{a}\gamma(b,\theta)u_*(b,\theta)\exp\left(-\int_b^a\left(\fF_*\gamma(s,\theta)+\alpha\right)\rd s\right)\rd b\nonumber\\
&=\phi(0,\theta)e^{-\alpha a}\frac{u_*(a,\theta)}{\fF_* \fS_*(\theta)}-\langle\lambda,\phi\rangle u_*(a,\theta)\int_{0}^{a}\gamma(b,\theta)e^{-\alpha(a-b)}\rd b\nonumber\\
&=\left(\frac{\langle\lambda,\phi\rangle}{\fF_*}+\frac{\langle \gamma K (\cdot,\theta),\phi\rangle}{\fS_*(\theta)}\right)u_*(a,\theta)e^{-\alpha a}-\langle\lambda,\phi\rangle u_*(a,\theta)\int_{0}^{a}\gamma(b,\theta)e^{-\alpha(a-b)}\rd b,\label{vvm-vl}
\end{align}
where we used in the second line Expression~\eqref{VVM-eq-st} of $u_*$, 
and \eqref{VVM-eq-c-1} in the last line.
We then deduce the following condition on the eigenvalues of the operator $\cA+F'(u_*)$ when there no memory of previous infection.

\begin{lemma}\label{lem:Lbda-int-equiv}
Suppose that Assumptions~\ref{Hyp:endemicity} and \ref{Hyp:endemic-stability} are satisfied. Let  $\alpha\in\dC$ be an eigenvalue of the operator $\cA+F'(u_*)$.   In the memory-free case ($K\equiv 1$ and $\nu$ probability measure), $\alpha$ satisfies 
\begin{align}
&\int_{\dR_+\times\Theta}\hskip-0.2cm u_*(a,\theta)e^{-\alpha a}\rd a\nu(\rd\theta)\notag
\\
&=\frac{\fF_*}{R_0}\int_{\dR_+\times\Theta}\hskip-0.2cm\lambda(a,\theta)e^{-\alpha a}\rd a\nu(\rd\theta)\int_{\dR_+\times\Theta}\hskip-0.2cm u_*(a,\theta)\int_{0}^{a}\gamma(b,\theta)e^{-\alpha(a-b)}\rd b\rd a\nu(\rd\theta).\label{eq:condition-sur-alpha}
\end{align}
\end{lemma}
\begin{proof} Recall that the operator $\cA+F'(u_*)$ is studied on the space defined in \eqref{eq-dom-A'}. Let $\phi$ be an  associated  eigenfunction to $\alpha$ with $\int_{\dR_+\times\Theta}\phi(a,\theta)\rd a\nu(\rd\theta)=0$. 

Since $K\equiv 1$, we have $\fS_*\equiv\frac{1}{R_0}$ by Example~\ref{ex:values of S_*}.\ref{ex:no memory} with $R_0=\dE_\nu\SBRA{\int_0^\infty\lambda(a)\rd a}$, and Equation~\eqref{vvm-vl} satisfied by the eigenfunction $\phi$ becomes
\begin{align}\label{eq:phi-no-memory}
\phi(a,\theta)
&=\PAR{\frac{\langle\lambda,\phi\rangle}{\fF_*}+R_0\langle \gamma,\phi\rangle}u_*(a,\theta)e^{-\alpha a}-\langle\lambda,\phi\rangle u_*(a,\theta)\int_{0}^{a}\gamma(b,\theta)e^{-\alpha(a-b)}\rd b.
\end{align}
We deduce that
\begin{align*}
\SCA{\lambda,\phi}
&=\PAR{\frac{\langle\lambda,\phi\rangle}{\fF_*}+R_0\langle \gamma,\phi\rangle}\int_{\dR_+\times\Theta}\lambda(a,\theta)u_*(a,\theta)e^{-\alpha a}\rd a\nu(\rd\theta),
\end{align*}
using the fact that $\lambda$ and $\gamma$ have disjoint supports by Assumption~\ref{Hyp:endemicity}.\ref{Hyp:disjoint-support}.
We deduce that if $\SCA{\lambda,\phi}=0$, then $\SCA{\gamma,\phi}=0$ and $\phi\equiv 0$ (by \eqref{eq:phi-no-memory}), which contradicts the fact that $\phi$ is an eigenfunction. 
Consequently, $\langle\lambda,\phi\rangle\neq0$ and we have
\begin{align}
&1-\frac{1}{\fF_*}\int_{\dR_+\times\Theta}\hskip-0.4cm\lambda(a,\theta)u_*(a,\theta)e^{-\alpha a}\rd a\nu(\rd\theta)
 =R_0\frac{\SCA{\gamma,\phi}}{\langle\lambda,\phi\rangle}\int_{\dR_+\times\Theta}\hskip-0.4cm\lambda(a,\theta)u_*(a,\theta)e^{-\alpha a}\rd a\nu(\rd\theta).\label{Lbda-int}
\end{align}
On the other hand, we have from \eqref{eq:phi-no-memory}
\begin{multline*}
\int_{\dR_+\times\Theta}\phi(a,\theta)\rd a\nu(\rd\theta)=\PAR{\frac{\langle\lambda,\phi\rangle}{\fF_*}+R_0\langle \gamma,\phi\rangle}\int_{\dR_+\times\Theta}u_*(a,\theta)e^{-\alpha a}\rd a\nu(\rd\theta)\\
-\langle\lambda,\phi\rangle\int_{\dR_+\times\Theta}u_*(a,\theta)\int_{0}^{a}\gamma(b,\theta)e^{-\alpha(a-b)}\rd b\rd a\nu(\rd\theta).
\end{multline*}
As $\phi$ is such that  $\displaystyle{\int_{\dR_+\times\Theta}\phi(a,\theta)\rd a\nu(\rd \theta)=0}$,
it follows that
\begin{align}
&\langle\lambda,\phi\rangle\left(\int_{\dR_+\times\Theta}\hskip-0.4cm u_*(a,\theta)\int_{0}^{a}\gamma(b,\theta)e^{-\alpha(a-b)}\rd b\rd a-\frac{1}{\fF_*}\int_{\dR_+\times\Theta}\hskip-0.4cm u_*(a,\theta)e^{-\alpha a}\rd a\nu(\rd\theta)\right)\notag
\\
&\hskip 2cm=R_0\SCA{\gamma,\phi}\int_{\dR_+\times\Theta}u_*(a,\theta)e^{-\alpha a}\rd a\nu(\rd\theta).\label{Int-vvm}
\end{align}
Combining \eqref{Lbda-int} and \eqref{Int-vvm}, we deduce
\begin{multline*}
\left(1-\frac{1}{\fF_*}\int_{\dR_+\times\Theta}\lambda(a,\theta)u_*(a,\theta)e^{-\alpha a}\rd a\nu(\rd\theta)\right)\int_{\dR_+\times\Theta}u_*(a,\theta)e^{-\alpha a}\rd a\nu(\rd\theta)\\
=\left(\int_{\dR_+\times\Theta}u_*(a,\theta)\int_{0}^{a}\gamma(b,\theta)e^{-\alpha(a-b)}\rd b\rd a-\frac{1}{\fF_*}\int_{\dR_+\times\Theta}u_*(a,\theta)e^{-\alpha a}\rd a\nu(\rd\theta)\right)\\\times\int_{\dR_+\times\Theta}\lambda(a,\theta)u_*(a,\theta)e^{-\alpha a}\rd a\nu(\rd\theta),
\end{multline*}

which can be simplified in the following way
  \begin{align*}
&\int_{\dR_+\times\Theta}u_*(a,\theta)e^{-\alpha a}\rd a\nu(\rd\theta)\notag
\\
&=\int_{\dR_+\times\Theta}\lambda(a,\theta)u_*(a,\theta)e^{-\alpha a}\rd a\nu(\rd\theta)\int_{\dR_+\times\Theta}u_*(a,\theta)\int_{0}^{a}\gamma(b,\theta)e^{-\alpha(a-b)}\rd b\rd a\nu(\rd\theta)\notag
\\
&=\frac{\fF_*}{R_0}\int_{\dR_+\times\Theta}\lambda(a,\theta)e^{-\alpha a}\rd a\nu(\rd\theta)\int_{\dR_+\times\Theta}u_*(a,\theta)\int_{0}^{a}\gamma(b,\theta)e^{-\alpha(a-b)}\rd b\rd a\nu(\rd\theta),
\end{align*}
thank to Expression~\eqref{VVM-eq-st} of $u_*$ and $\mathrm{Supp}(\lambda)\cap \mathrm{Supp}(\gamma)=\emptyset$ in the last line.
\end{proof}

We now deduce the local stability of the equilibrium. 

\begin{proof}[Proof of Theorem~\ref{VVM-conj}]

By Lemma~\ref{lem:bound on w_ess}, we know that $w_{\rm ess}$ is smaller than the negative value $-\fF_*\sigma$ and by assumption there is no eigenvalue $\alpha$ of the operator $\cA+F'(u_*)$ with $\cR e(\alpha)\ge 0$. Consequently, by the relation \eqref{eq:relation w_0}, we deduce that
$
w_0(\cA+F'(u_*))<0$.
The conclusion of Theorem~\ref{VVM-conj} follows by  taking $w_0:=w_0(\cA+F'(u_*))$, and applying Thieme~\cite[Theorem 4.2]{thieme1990semiflows}  to Equation~\eqref{VVM-eq-weak-r}.
\end{proof}

\begin{remark} 
We assume that assumptions of Theorem~\ref{VVM-conj} are satisfied.
\begin{enumerate}
\item
 If the operator $\cA_*+F'(u_*)$ has no eigenvalue $\alpha$ with real part $\cR e(\alpha)>-\sigma\fF_*$, then
    the result of Theorem~\ref{VVM-conj} holds for any $w\in\PAR{-\sigma\cF_*,0}$.
    \item When $a\mapsto\gamma(a,.)$ are non decreasing functions, using Equation~\eqref{VVM-eq-st},
    we have
    \begin{align*}
\fF_*\int_{\dR_+\times \Theta} u_*(a,\theta)\PAR{\int_0^a\gamma(b,\theta)\rd b}\rd a\nu(\rd \theta)
\le 1-\frac{\fF_*}{R_0}\dE_\nu[T_I]<1,
\end{align*}
where $T$ is a nonnegative variable on $\PAR{\Theta, \cH,\nu}$ such that $\mathrm{Supp}(\gamma(.,\theta))\subset[T(\theta), +\infty)$. Consequently,
    Condition~\eqref{eq:condition-sur-alpha} is not satisfied when $\alpha\to 0$. This means that there is no eigenvalue of the operator $\cA+F'(u_*)$ close to $0$.
\end{enumerate}
\end{remark}

In the next section, we study the local stability for a specific model.

\subsection{Local stability of endemicity for  a SIS-type model}\label{end-spec-L}
Throughout this section, we assume that $K\equiv 1$ and $\nu$ is a probability measure. 
We consider bounded infectivity curves $a\mapsto \lambda(a,.)$ with support in $[0,T_I(.)]$, and 
step susceptibility curves $\gamma$ of the form
 \[
    \gamma(a,.)=\ind_{[T_I(.),+\infty)}(a),\quad \text{for $a\ge 0$},
    \]
    where $T_I$ is a positive integrable variable defined on the probability space $\PAR{\Theta,\cH, \nu}$. 
We can compute each quantity explicitly:  $\gamma_*\equiv 1$, $\fS_*\equiv \frac{1}{R_0}$, $\fF_*=\frac{R_0-1}{\dE_\nu[T_I]}$ from~\eqref{VVM-eq-H}, and for $(a,\theta)\in\dR_+\times \Theta$, we have from Equation~\eqref{VVM-eq-st}
    \[
    u_*(a,\theta)=\frac{\fF_*}{R_0}\ind_{a< T_I(\theta)}+\frac{\fF_*}{R_0}\re^{-\fF_*(a-T_I(\theta))}\ind_{a\ge T_I(\theta)}.
    \]
    We assume $R_0>1$ to ensure the existence of the endemic equilibrium.
    We easily compute that for $\alpha\neq 0$
    \[
    \int_{\dR_+\times \Theta} u_*(a,\theta)\re^{-\alpha a}\rd a\nu(\rd \theta)=\frac{\fF_*}{\alpha R_0}-\frac{\fF_*^2}{\alpha R_0(\fF_*+\alpha)}\dE_\nu\SBRA{\re^{-\alpha T_I}}.
    \]
    On the other hand, for $\alpha\neq 0$, we have
    \begin{align*}
    &\int_{\dR_+\times \Theta}\lambda(a,\theta)u_*(a,\theta)\re^{-\alpha a}\rd a\nu(\rd \theta)=\fF_*\frac{\int_0^\infty\dE_\nu[\lambda(a)]\re^{-\alpha a}\rd a}{R_0}\\
    &\int_{\dR_+\times \Theta} u_*(a,\theta)\PAR{\int_0^a\gamma(b,\theta)\re^{-\alpha(a-b)}\rd b}\rd a\nu(\rd \theta)
    =\frac{1}{R_0(\fF_*+\alpha)}.
    \end{align*}
        Then Condition~\eqref{eq:condition-sur-alpha} is equivalent to
    \begin{align*}
   1+ \frac{\fF_*}{\alpha}\PAR{1-\dE_\nu\SBRA{\re^{-\alpha T_I}}}&= \frac{\int_0^\infty\dE_\nu[\lambda(a)]\re^{-\alpha a}\rd a}{R_0},
    \end{align*}
    which can be rewritten, using $\int_0^{T_I}\re^{-\alpha a}\rd a=\frac{1}{\alpha}\PAR{1-\re^{-\alpha T_I}}$,
    \begin{equation}\label{eq:condi-alpha-exple}
    1+\fF_*\int_0^\infty\re^{-\alpha a}\nu(T_I>a)\rd a= \frac{\int_0^\infty\dE_\nu[\lambda(a)]\re^{-\alpha a}\rd a}{R_0}.
    \end{equation}
    As $R_0>1$, we observe that \eqref{eq:condi-alpha-exple} cannot be satisfied when $\alpha\to 0$, and then the operator $\cA+F'(u_*)$ cannot have eigenvalues close to $0$.
We now focus on a SIS-type model, for which we can conclude. 
Since the local stability stated in Theorem~\ref{VVM-conj} requires Assumption~\ref{Hyp:endemic-stability} to hold, the classical SIS model does not fall within the scope of this result, as it assumes that the infection duration $T_I$ follows an exponential distribution and is therefore unbounded.
However, to our knowledge, this is the first result of endemic equilibrium stability for this type of model.
\begin{proposition}\label{prop:exple-stability}
    We consider a memory-free model ($K\equiv 1)$. Let $a_*>0$ and $T_I=E\wedge a_*$ with $E$ an exponential variable, defined on the probability space $\PAR{\Theta,\cH,\nu}$, of parameter $\rho$: $\nu(E>a)=\re^{-\rho a}$. We consider step infectivity  and susceptibility  curves:
 \[
   \lambda(a,.)=\lambda_*\ind_{[0,T_I(.))}(a), \quad  \gamma(a,.)=\ind_{[T_I(.),+\infty)}(a),\quad \text{for $a\ge 0$.}
    \]
We assume that $\lambda_*\le 2\rho\re^{\rho a_*}$. Then, there is a unique endemic equilibrium when 
$ \lambda_*\dE_\nu[T_I]>1$,
   which is locally stable. 
\end{proposition}
When $a_*=+\infty$, the model presented in Proposition~\ref{prop:exple-stability} coincides with the classical SIS model. The condition on the parameters is not too restrictive, because for fixed values of $\lambda_*$ and $\rho$, it is sufficient to choose $a_*$ large enough to satisfy it.

\begin{proof} The assumptions of Theorem~\ref{thm:existence-endemicity} are satisfied, and since $\gamma$ is non decreasing, there is a unique endemic equilibrium when $ R_0=\dE_\nu\SBRA{\int_0^\infty\lambda(a)\rd a}=\lambda_*\dE_\nu[T_I]=\lambda_* \frac{1-\re^{-\rho a_*}}{\rho}>1$.  To obtain local stability by Theorem~\ref{VVM-conj}, we prove by contradiction that Condition~\eqref{eq:condi-alpha-exple} cannot be satisfied for $\alpha\in\dC$ with $\cR e(\alpha)\ge 0$.

Let $\alpha=x+iy$ with $x\ge 0$. Computing the real 
    part of \eqref{eq:condi-alpha-exple}, we obtain
    \begin{equation*} 
1+\fF_*\int_0^{\infty}\re^{-x a}\cos(ya)\nu(T_I>a)\rd a= \frac{\int_0^\infty\dE_\nu[\lambda(a)]\re^{-x a}\cos(ay)\rd a}{R_0}.
    \end{equation*}
    We note that $\ABS{\frac{\int_0^\infty\dE_\nu[\lambda(a)]\re^{-x a}\cos(ay)\rd a}{R_0}}\le 1$. 
As $\lambda(a,.)=\lambda_*\ind_{T_I(.)> a}$, we have $\fF_*=\lambda_*\left(1-\frac{1}{R_0}\right).$ Thus the real part of Equation~\eqref{eq:condi-alpha-exple} satisfies
\begin{equation}\label{sis-V}
1+(R_0-2)\frac{\lambda_*}{R_0}\int_0^{\infty}\re^{-x a}\cos(ya)\nu(T_I>a)\rd a=0.
\end{equation}
We first study the case $1<R_0<3$. Then, for $x\ge0$, since $R_0=\lambda_*\dE_\nu[T_I]$,
\[\left|(R_0-2)\frac{\lambda_*}{R_0}\int_0^{\infty}\re^{-x a}\cos(ya)\nu(T_I>a)\rd a\right|\le \ABS{2-R_0}<1.\] Consequently,
$
1+(R_0-2)\frac{\lambda_*}{R_0}\int_0^{\infty}\re^{-x a}\cos(ya)\nu(T_I>a)\rd a>0$.

On the other hand, when $R_0\ge 3$, 
    we  compute
\begin{align*}
\int_0^{\infty}\re^{-x a}\cos(ya)\nu(T_I>a)\rd a&=
 \int_0^{a_*}\re^{-x a}\cos(ya)\re^{-\rho a}\rd a\\
&\hskip -2cm=\frac{x+\rho}{(x+\rho)^2+y^2}+\frac{\re^{-(x+\rho)a_*}}{(x+\rho)^2+y^2}\PAR{y\sin(a_*y)-(x+\rho)\cos(a_*y)}.
 \end{align*}
We have
\begin{align*}
&1+(R_0-2)\frac{\lambda_*}{R_0}\int_0^{\infty}\re^{-x a}\cos(ya)\nu(T_I>a)\rd a\\
&=1+\frac{(R_0-2)y\sin(a_*y)\re^{-(x+\rho)a_*}}{R_0((x+\rho)^2+y^2)}+\frac{\lambda_*(R_0-2)(x+\rho)}{R_0((x+\rho)^2+y^2)}\PAR{1-\cos(a_*y)\re^{-(x+\rho)a_*}}.		
\end{align*}
However, for $R_0\ge 3$, $x\ge 0$ and $y\in\dR$,
\begin{align*}
&\frac{\lambda_*(R_0-2)}{R_0}\frac{x+\rho}{(x+\rho)^2+y^2}\PAR{1-\cos(a_*y)\re^{-(x+\rho)a_*}}\ge \frac{\lambda_*\rho\PAR{1-\re^{-\rho a_*}}}{R_0\PAR{(x+\rho)^2+y^2}}>0,\\
&1+\lambda_*\frac{R_0-2}{R_0}\frac{y\sin(a_*y)\re^{-(x+\rho)a_*}}{(x+\rho)^2+y^2}\ge1-\lambda_*\frac{\ABS{y}\re^{-\rho a_*}}{\rho^2+y^2}\ge 1-\frac{\lambda_*\re^{-\rho a_*}}{2\rho}.
\end{align*}
Consequently, as soon as $\frac{\lambda_*\re^{-\rho a_*}}{2\rho}\le 1$, Equation~\eqref{sis-V} has no solution $\alpha=x+iy$ with $x\ge 0$.
In conclusion, by Theorem~\ref{VVM-conj}, there is local stability of the equilibrium.
\end{proof}

\section{Application to models incorporating a vaccination policy}\label{sec:vaccin}

In this section, we apply our results on the existence of endemic equilibria to specific cases of susceptibility curves taking into account a vaccine policy into the model. We will consider two types of vaccine policy. The first one is a toy model with memory of the last infection where the results are explicit, and the second one is a more complex model, but without memory of the previous infections, similar to the one studied in \cite{foutel2023optimal}.   For modelling purposes, throughout this section we assume that $\nu$ is a probability measure.

\subsection{One shot of vaccination after an infection}\label{sec:vaccin-one-shot}

We consider the case where access to the vaccine is very restricted, and only people vulnerable to the disease (those who have been infected) have the opportunity to receive a single dose of vaccine to prevent re-infection. It is a generalization of the model introduced in Example~\ref{exple:monotony-H}. 

Let $T_I$ be the infection duration after a contamination defined on $(\Theta,\cH,\nu)$.  We introduce two integrable positive random times $T_R$ and $T_V$ on $(\Theta,\cH,\nu)$, with $T_R\ge T_I$ and $T_V\ge T_I$ $\nu$-a.e, where $T_R-T_I$ is the immunity period after an infection, and $T_V$ is the vaccination time after an infection. Let $\alpha,\beta\in(0,1]$ be random variables on $(\Theta,\cH,\nu)$ independent of $T_R$ and $T_V$.
We consider susceptibility curves given by, for $(a, \theta)\in\dR_+\times \Theta$,
\begin{align*}
\gamma(a,\theta)=
\begin{cases}
\alpha(\theta)\ind_{T_R(\theta)\le a<T_V(\theta)}+\beta(\theta)\ind_{a\ge T_V(\theta)}&\text{ if $T_V(\theta)>T_R(\theta)$,}\\[0.2cm]
\beta(\theta)\ind_{a\ge T_V(\theta)}&\text{ otherwise.}
\end{cases}
\end{align*}

We assume that Assumptions~\ref{Hyp:Kernel} and \ref{Hyp:endemicity}.\ref{Hyp:infectivity and susceptibility} are satisfied. We also assume that $\beta$ is bounded from below by a positive constant $\nu$-a.e.,
which implies that Assumption~\ref{Hyp:endemicity} is satisfied with $\gamma_*\equiv \beta$ for a good positive memory kernel $K$. We introduce the expectation $\dE^*_\nu$ defined in Remark~\ref{rk:endemic-conditon} and $\kappa=\dE^*_\nu\SBRA{\fS_*}$.

When $T_V(\theta)\le T_R(\theta)$, we have
\[
x\int_0^\infty  \re^{-x\int_0^a\gamma(s,\theta)\rd s}\rd a=xT_V(\theta)+\frac{1}{\beta(\theta)}.
\]

When $T_V(\theta)>T_R(\theta)$, we have
\begin{align*}
x\int_0^\infty  \re^{-x\int_0^a\gamma(s,\theta)\rd s}\rd a=xT_R(\theta)+\frac{1}{\alpha(\theta)}-\re^{-x\alpha(\theta)(T_V(\theta)-T_R(\theta))}\PAR{\frac{1}{\alpha(\theta)}-\frac{1}{\beta(\theta)}}.
\end{align*}
We note that the function $H$, defined by \eqref{VVM-eq-H}, is increasing when $\alpha\le \beta$ $\nu$-a.e., which is indeed obvious since $\gamma$ is non-decreasing.

Let us consider the case $\alpha\ge \beta$ $\nu$-a.e., i.e. $\gamma$ non-monotone, which is realistic when  vaccination improves immunity.
We have
\[
H(x)=\kappa\PAR{\dE_\nu^*\SBRA{\frac{1}{\alpha}}+x\dE_\nu^*[T_R\wedge T_V]-\dE_\nu^*\SBRA{\PAR{\frac{1}{\alpha}-\frac{1}{\beta}}\re^{-x\alpha(T_V-T_R)_+}}},
\]
with $a\wedge b=\min(a,b)$ and $a_+=\max(a,0)$.
We easily compute
\[
H'(x)=\kappa\PAR{\dE_\nu^*[T_R\wedge T_V]+\dE_\nu^*\SBRA{\PAR{1-\frac{\alpha}{\beta}}(T_V-T_R)_+\re^{-x\alpha(T_V-T_R)_+}}}.
\]
If $\alpha\ge \beta$ $\nu$-a.e., $H'$ is increasing, with \[
H'(0)=\frac{1}{R_0}\PAR{\dE_\nu^*[T_R\wedge T_V]+\dE_\nu^*\SBRA{1-\frac{\alpha}{\beta}}\dE_\nu^*\SBRA{(T_V-T_R)_+}}.
\]
Consequently, 
\begin{itemize}
\item
if $\alpha\ge \beta$ $\nu$-a.e. and $\dE_\nu^*[T_R\wedge T_V]\ge\dE_\nu^* \SBRA{\frac{\alpha}{\beta}-1}\dE_\nu^*\SBRA{(T_V-T_R)_+}$, then $H$ is increasing and by Theorem~\ref{thm:existence-endemicity}, there is a (unique) endemic equilibrium if and only if $\dE_\nu^*\SBRA{\frac{1}{\beta}}<\frac{1}{\kappa}=R_0^*$, where $R_0^*$ is defined by \eqref{eq:R_0^*}.

\item if  $\alpha\ge \beta$ $\nu$-a.e. and $\dE_\nu^*[T_R\wedge T_V]< \dE_\nu^*\SBRA{\frac{\alpha}{\beta}-1}\dE_\nu^*\SBRA{(T_V-T_R)_+}$, then $H$ is decreasing and then increasing, with $H(0)=\kappa\dE_\nu^*\SBRA{\frac{1}{\beta}}$.

Let $x_{\min}$ such that $H'(x_{\min})=0$. Then if  $\dE_\nu^*\SBRA{\frac{1}{\beta}}<\frac{1}{\kappa}$, there is a (unique) endemic equilibrium. When $\dE_\nu^*\SBRA{\frac{1}{\beta}}>\frac{1}{\kappa}$ and $H(x_{\min})>1$, there is extinction of the disease. But when $\dE_\nu^*\SBRA{\frac{1}{\beta}}>\frac{1}{\kappa}$ and $H(x_{\min})<1$, there exist two solutions to $H(x)=1$, and then two possible endemic equilibria.
\end{itemize}

This is an interesting toy model because we can exhibit the existence of more than endemic equilibrium under good conditions.
Note that for susceptibility curves $\gamma$ such that $\alpha\ge\beta$ and  $T_V>T_R$, Condition~\eqref{eq:condi-uniq-endemic}, which ensures the uniqueness of the endemic equilibrium when it exists, is not satisfied.

\subsection{A renewal process of vaccination}\label{sec:vaccin-renewal}

In this section we assume that  there is no memory of the previous infections ($K\equiv 1$ and $\nu$ probability measure). We also assume that the vaccine policy is independent of the evolution of the disease for each individuals.
To model the independence between the vaccine policy and the disease, we assume that the probability space  $(\Theta,\cH,\nu)$ is a product space $\Theta=\Theta_1\times \Theta_2$ with $\nu=\nu_1\otimes \nu_2$.

As in \cite{foutel2023optimal}, we consider the case where between infections the individuals are regularly vaccinated. More precisely, after an infection, successive times of vaccination are sampled for the infected individual according to a renewal process. At each new infection of the individual, new times of vaccination are sampled independently of the previous ones. We assume that the first vaccination occurs after the end of the individual's infectious period.

We introduce $\sigma$ a function defined on $\dR_+\times\Theta_1$ with values in $[0,1]$. This function models the susceptibility between two vaccine doses. We also consider $T_I$ and $T_V$, two non-negative integrable functions defined respectively on $\Theta_1$ and $\Theta_2$,  modeling the duration of infection after contamination for the former, and the duration between two vaccine injections for the latter.

At each new infection, we sample $\theta=(\theta_1,\theta_2)$ where $\theta_1=(\theta_{1,n})_{n\ge 0}$ and $\theta_2=(\theta_{2,n})_{n\ge 1}$ are two sequences of i.i.d. random variables with respective distribution $\nu_1$ and $\nu_2$.
We then consider susceptibility curves of the following form, for $a\ge 0$,
\[
\gamma(a,\theta)=\sum_{n=0}^\infty\sigma_n(a-\tau_n(\theta))\ind_{\tau_n(\theta)\le a<\tau_{n+1}(\theta)}=\sum_{n=0}^{N(a,\theta)}\sigma_n(a-\tau_n(\theta))\ind_{\tau_n(\theta)\le a<\tau_{n+1}(\theta)},
\]
where, for $n\ge 0$,
\begin{itemize}
\item $\sigma_n(a)=\sigma(a,\theta_{1,n})$
\item the times $\tau_n(\theta)$ are defined by induction by $\tau_0(\theta)=T_I(\theta_{1,0})$, and for $n\ge 0$
\[
\tau_{n+1}(\theta)=\tau_{n}(\theta)+T_{V,n+1}(\theta),
\]
with $T_{V,n}(\theta)=T_{V}(\theta_{2,n})$.
\item $N$ is the counting process associated with the times $(\tau_n)_{n\ge 0}$: $N(a,\theta)=\sum_{n\ge 1}\ind_{\tau_n(\theta)\le a}$.
\end{itemize}
We note that the process $(N(a,\cdot))_{a\ge0}$ is a renewal process (see, e.g. \cite{asmussen2003applied}). 

This model with a renewal vaccination policy is the same as that studied in \cite{foutel2023optimal}, except that the duration of infection after contamination $T_I(\theta_{1,0})$ and the first susceptibility curve $\sigma_0(\cdot)=\sigma(\cdot,\theta_{1,0})$ are not necessarily independent and that $a\mapsto\sigma(a,\cdot)$ are not necessarily monotone. Let us compute the endemic threshold.
To this end, we first derive the function $\gamma_*$. We have, for $b>0$,
\begin{align*}
\frac{1}{b}\int_0^b\gamma(a,\theta)\rd a&=
\frac{1}{b}\sum_{n=1}^{N(b,\theta)}\int_{0}^{T_V(\theta_{2,n})}\sigma(a,\theta_{1,n-1})\rd a+\frac{1}{ b}\int_{0}^{b-\tau_{N(b,\theta)}(\theta)}\sigma(a,\theta_{1,N(b,\theta)})\rd a,
\end{align*}
with $\displaystyle{\frac{1}{ b}\int_{0}^{b-\tau_{N(b,\theta)}(\theta)}\sigma_{N(b,\theta)}(a)\rd a\le \frac{b-\tau_{N(b,\theta)}(\theta)}{b}}$. 
We also note that $$\displaystyle{\frac{b-\tau_{N(b,\theta)}(\theta)}{b}=1-\frac{N(b,\theta)}{b}\frac{1}{N(b,\theta)}\sum_{k=1}^{N(b,\theta)}T_V(\theta_{2,k})}.$$ 
A classical result on renewal processes gives   (see \cite[Proposition 1.4, page~$140$]{asmussen2003applied})
\[
\frac{N(b,\theta)}{b}\underset{b\to +\infty}{\overset{a.s.}{\longrightarrow}}\PAR{\int_{\Theta_2}T_V(\theta_2)\nu_2(\rd\theta_2)}^{-1}.
\]
By assumptions $\PAR{\displaystyle{\int_0^{T_V(\theta_{2,n})}\sigma(a,\theta_{1,n-1})\rd a}}_{n\ge 1}$ are integrable i.i.d. random variables, consequently the limit $\displaystyle{\gamma_*=\lim_{b\to +\infty}\frac{1}{b}\int_0^b\gamma(a,\theta)\rd a}$ exists a.e. and
\[
\gamma_*=\frac{\int_{\Theta_1\times\Theta_2}\left(\int_0^{T_V(\theta_{2})}\sigma(a,\theta_{1})\rd a\right)\nu_1(d\theta_1)\nu_2(d\theta_2)}{\int_{\Theta_2}T_V(\theta_2)\nu_2(\rd\theta_2)}\quad \text{a.e.}
\]
We note that  $\gamma_*$ is a constant. As we are in the memory-free case, we know by Example~\ref{ex:values of S_*}.\ref{ex:no memory} that $\fS_*=1/R_0$.
By Theorem~\ref{thm:existence-endemicity}, there is an endemic equilibrium if
\[
R_0>\frac{1}{\gamma_*}=\frac{\int_{\Theta}T_V(\theta_2)\nu_2(\rd\theta_2)}{\int_{\Theta}\int_{\Theta}\left(\int_0^{T_V(\theta_{2})}\sigma(a,\theta_{1}))\rd a\right)\nu_1(d\theta_1)\nu_2(d\theta_2)}.
\] 
We easily remark that we can rewrite $\frac{1}{\gamma_*}=\dE_\nu\SBRA{T_V}/\dE_\nu\SBRA{\int_0^{T_V}\sigma(a)\rd a}$. We then recover the threshold obtained by \cite{foutel2023optimal} for a similar model with the same vaccine policy.

\subsection*{Acknowledgements}
This project has benefited from discussions with Bertrand Cloez, Raphaël Forien, and Étienne Pardoux, whom we would like to thank for their enthusiasm and generosity. The authors also thank the anonymous referee for their comments, which helped improve the paper.

\appendix
\section{Absolute continuity of the solution}\label{A:abs-conti}
\renewcommand{\theequation}{\Alph{section}.\arabic{equation}}
\let \sappend=\section
\renewcommand{\section}{\setcounter{equation}{0}\sappend}

\begin{proposition}\label{AP-density}
	Let $\mu$ be the solution to Equation~\eqref{VVM-eq1}. If $\mu_0(\rd a,\rd\theta)$ is absolutely continuous with respect to the measure on $\dR_+\times\Theta,$ then, for every $t\ge0$, $\mu_t(\rd a,\rd\theta)$ is also absolutely continuous with respect to $\rd a\nu(\rd \theta)$. 
\end{proposition}
\begin{proof}
We denote by $u_0$ the density of $\mu_0$ with respect to $\rd a\nu(\rd \theta)$: $\mu_0(\rd a,\rd\theta)=u_0(a,\theta)\rd a\nu(\rd \theta)$.
Let $\varphi$ be a nonnegative measurable function of class $\cC^1$ with respect to its first variable on $\dR_+\times \Theta$. 
We define
\begin{equation*}
	\forall t\in\dR_+,\forall s\in[0,t],\forall (a,\theta)\in\dR_+\times\Theta,\quad f(s,a,\theta)=\varphi(a-(s-t),\theta).
\end{equation*}
As in the proof of Proposition~\ref{VVM-ex-th} from \eqref{eq:Rjump-term} and \eqref{VVM-eq1}, denoting $\varphi_s(a,\theta)=\varphi( a+s,\theta)$,
\begin{equation*}
	\langle\mu _t,\varphi\rangle=\langle\mu _0,\varphi_t\rangle+\int_{0}^{t}\langle\mu _{s},\lambda\rangle\langle\mu _{s},R\varphi_{t-s}\rangle\rd s.
\end{equation*} 
Therefore, using the fact that $\varphi$ is non-negative,
\begin{align*}
	\langle\mu _t,\varphi\rangle&\le\int_{\dR_+\times\Theta}\varphi_t(a,\theta)\mu_0(\rd a,\rd\theta)\\
    &\hskip 0.5cm +\int_{0}^{t}\langle\mu _{s},\lambda\rangle\int_{\dR_+\times\Theta}\int_\Theta\varphi(t-s,\widetilde{\theta})\gamma( a,\theta)K(\theta,\widetilde\theta)\nu(\rd\widetilde\theta)\mu_s(\rd a,\rd\theta)\rd s\\
	&=\int_t^\infty\int_{\Theta}\varphi(s,\theta)u_0(s-t,\theta)\rd s\nu(\rd\theta)\\
    &\hskip 0.5cm+\int_{0}^{t}\langle\mu _{t-s},\lambda\rangle\int_{\dR_+\times\Theta}\int_\Theta\varphi(s,\widetilde{\theta})\gamma( a,\theta)K(\theta,\widetilde\theta)\nu(\rd\widetilde\theta)\mu_{t-s}(\rd a,\rd\theta)\rd s\\
	&=\int_t^\infty\int_{\Theta}\varphi(s,\widetilde\theta)u_0(s-t,\widetilde\theta)\rd s\nu(\rd\theta)
    \\
    &\hskip 0.5cm+\int_{0}^{t}\int_{\Theta}\varphi(s,\widetilde{\theta})\langle\mu _{t-s},\lambda\rangle\int_{\dR_+\times\Theta}\gamma( a,\theta)K(\theta,\widetilde\theta)\mu_{t-s}(\rd a,\rd\theta)\rd s\nu(\rd\widetilde\theta),\\
	&=\int_{\dR_+\times\Theta}\varphi(s,\widetilde\theta)H_t(s,\widetilde\theta)\rd s\nu(\rd\widetilde\theta),
\end{align*}
where 
\[H_t(s,\widetilde\theta)=u_0(s-t,\widetilde\theta)\indic{s> t}+\indic{s\le t}\langle\mu _{t-s},\lambda\rangle\int_{\dR_+\times\Theta}\gamma( a,\theta)K(\theta,\widetilde\theta)\mu_{t-s}(\rd a,\rd\theta).\]
By density, we deduce that for all bounded non-negative measurable function $\varphi$ on $\dR_+\times\Theta$, we have
\begin{equation*}
	\langle\mu _t,\varphi\rangle\le\int_{\dR_+\times\Theta}\varphi(s,\widetilde\theta)H_t(s,\widetilde\theta)\rd s\nu(\rd\widetilde\theta).
\end{equation*}  
We conclude with the Radon-Nikodym theorem.
\end{proof}

\section{Operators compactness}\label{A:compact-operators}

The compactness proofs of the operators studied in Section~\ref{sec:local stability} are based on the Riez-Fréchet-Kolmogorov criterion, see  \cite[Theorem~4.26]{brezis2011analyse}. To  do so, we need to assume that $\Theta$ is an open subset of $\dR^d$ and $\nu$ is a $\sigma$-finite measure absolutely continuous with respect to the Lebesgue measure. 

\begin{lemma}\label{lem:B*_compact}
The operator $\cB_*$ defined in \eqref{eq:def_A*-B*} is a linear compact operator on $\dX=L^1\PAR{\theta,\nu}\times  L^1(\rd a\otimes\nu)$ for the norm $\NRM{.}_\dX$, given by \eqref{eq:norm-X}.
\end{lemma}

\begin{proof}
Let us recall that $\cB_*$ is defined for $\phi\in  L^1(\rd a\otimes\nu)$ and $(a,\theta)\in \dR_+\times \Theta$ by
\[
\cB_*\begin{pmatrix}
0\\
\phi
\end{pmatrix}
(a,\theta)=\begin{pmatrix}
\fS_*(\theta)\langle\lambda,\phi\rangle+\SCA{\gamma K(\cdot,\theta), \phi}\fF_*\\-\gamma(a,\theta)u_*(a,\theta)\langle\lambda,\phi\rangle
\end{pmatrix}.
\]
For a function $f$ defined on $\dR_+\times \Theta$, and $(h,k)\in\dR\times \dR^d$, we define $f_{(h,k)}(a,\theta)=f(a+h,\theta+k)$, if $(a+h,\theta+k)\in\dR_+\times \dR^d$, and $f_{(h,k)}(a,\theta)=0$ otherwise. 
By Assumptions~\ref{Hyp:endemicity}.\ref{Hyp:infectivity and susceptibility}, the fact that $u_*$ is a density with respect to $\rd a\otimes\nu$ on $\dR_+\times \Theta$, and $\fF_*\le \lambda_*$ (see Equation~\eqref{VVM-eq-st}), we easily observe that for $\phi\in   L^1(\rd a\otimes\nu)$ with $\NRM{\phi}_{ L^1(\rd a\otimes\nu)}\le 1$,
\begin{align*}
&\NRM{{\cB_*}_{(h,k)}\begin{pmatrix}
0\\
\phi
\end{pmatrix}
-\cB_*\begin{pmatrix}
0\\
\phi
\end{pmatrix}}_\dX
\le
\lambda_*\NRM{{\fS_*}_k-\fS_*}_{L^1(\nu)}+\lambda_*\NRM{\PAR{\gamma u_*}_{(h,k)}-\gamma u_*}_{ L^1(\rd a\otimes\nu)}
\\
&\hskip 2.5cm +\lambda_*\int_{\dR_+\times\Theta^2}\ABS{\phi(a,\widetilde\theta)}\ABS{K(\widetilde\theta, \theta+k)-K(\widetilde\theta,\theta)}\rd a\nu(\rd\widetilde\theta)\nu(\rd\theta)\\
&\le
\lambda_*\NRM{{\fS_*}_k-\fS_*}_{L^1(\nu)}+\lambda_*\NRM{\PAR{\gamma u_*}_{(h,k)}-\gamma u_*}_{ L^1(\rd a\otimes\nu)}
\\
&\hskip 2.5cm +\lambda_*\int_{\Theta}\sup_{\widetilde\theta\in\Theta}\ABS{K(\widetilde\theta, \theta+k)-K(\widetilde\theta,\theta)}\nu(\rd\theta).
\end{align*}
By \cite[Lemma 4.3]{brezis2011analyse}, for $N\ge 1$, if $f\in L^1(\dR^N)$ (endowed with the Lebesgue measure), 
\begin{equation}\label{eq:conv-L1-brezis}
    \lim_{h\to 0}\NRM{f_{h}-f}_{L^1(\dR^N)}=0.
\end{equation}
As $\nu$ is absolutely continuous with respect to the Lebesgue measure on $\dR^d$, as $\fS_*$, $\gamma u_*$, and by Assumption~\ref{Hyp:endemicity}.\ref{Hyp:kernel in L^1}, $\displaystyle{\theta\mapsto \sup_{\widetilde\theta\in\Theta}K(\widetilde\theta,\theta)}$  are integrable functions, we obtain by \eqref{eq:conv-L1-brezis} the convergence of their related $L^1$-norm holds.

We now  introduce
\[
\cK=\BRA{\cB_*\begin{pmatrix}
0\\
\phi
\end{pmatrix}: \phi\in  L^1(\rd a\otimes\nu) \text{ with }\NRM{\phi}_{ L^1(\rd a\otimes\nu)}\le 1},
\]
and, for a measurable set $\Omega\subset \dR_+\times\Theta$, $\cK_{|\Omega}$ denotes the set of the restrictions to $\Omega$ of the functions of $\cK$.
By the Riez-Fréchet-Kolmogorov criterion \cite[Theorem~4.26]{brezis2011analyse}, we have proved that the closure of the set $\cK_{|\Omega}$ is compact for any measurable set $\Omega$ with finite measure.

Besides, we easily observe that $\cK$ is a bounded set of $\dX$. By Assumptions~\ref{Hyp:Kernel} and \ref{Hyp:endemicity}.\ref{Hyp:infectivity and susceptibility}, we have for $\phi\in  L^1(\rd a\otimes\nu)$ and $r>0$,
\[
\int_r^\infty\int_\Theta\gamma(a,\theta)u_*(a,\theta)\SCA{\lambda, \phi}\rd a\nu(\rd\theta)\le \lambda_*\NRM{\phi}_{ L^1(\rd a\otimes\nu)}\int_r^\infty\int_\Theta u_*(a,\theta)\rd a\nu(\rd\theta).
\]
Thus, $\displaystyle{\lim_{r\to \infty}\int_r^\infty\int_\Theta\gamma(a,\theta)u_*(a,\theta)\SCA{\lambda, \phi}\rd a\nu(\rd\theta)=0}$ uniformly on $\phi$ with\\ $\NRM{\phi}_{ L^1(\rd a\otimes\nu)}\le 1$,
since $u_*$ is a density on $\dR_+\times \Theta$ with respect to the measure $\rd a\otimes \nu$.
Consequently, by \cite[Corollary 4.27]{brezis2011analyse}, we deduce that $\cK$ has compact closure in $\dX$. This implies that $\cB_*$ is a compact operator on $\dX$.
\end{proof}

\begin{lemma}\label{eq-compact-T-2}
For all $t\ge0,\,T_*^2(t)$, defined in \eqref{eq:T_* for a>t-ess}, is a compact operator on $L^1(\dR_+\times \Theta,\rd a\otimes\nu)$.
\end{lemma}
    \begin{proof}
Let $t\ge 0$. For $\phi\in L^1(\dR_+\times \Theta,\rd a\otimes\nu)$, we easily note that
\[
\ABS{T_*^2(t)\left(\phi\right)(a,\theta)}\le 
\mathds{1}_{a\ge t}\ABS{\int_{0}^{t}u_*(a-t+s,\theta)\langle \lambda,T_*(s)(\phi)\rangle\rd s}.
\]
So, let us introduce the operator $G$ defined on $ L^1(\dR_+\times \Theta,\rd a\otimes\nu)$ by
\[G(t)(\phi)(a,\theta)=\mathds{1}_{a\ge t}\int_{0}^{t}u_*(a-t+s,\theta)\langle \lambda,T_*(s)(\phi)\rangle \rd s.\]
It suffices to prove $G(t)$ is a compact operator to deduce that $T_*^2(t)$ is also a compact operator, because $L^1(\dR_+\times \Theta,\rd a\otimes\nu)$ is a Banach space. 
From \eqref{eq:T_* for a>t-ess}, we remark that
\begin{align*}
    \|T_*(t)(\phi)\|_1&\le\NRM{\phi}_{ L^1(\rd a\otimes\nu)}+\lambda_*\int_0^t\|T_*(s)(\phi)\|_1 \rd s\\
    &\le \NRM{\phi}_{ L^1(\rd a\otimes\nu)} \re^{\lambda_* t},
\end{align*}
since $\ABS{\langle \lambda,T_*(t)(\phi)\rangle}\le\lambda_*\|T_*(t)(\phi)\|_1,\,\gamma\le1$, $u_*$ is a probability density on $\dR_+\times \Theta$, and
 where the last inequality is a consequence of Gronwall's inequality. Using the same notations as in the proof of Lemma~\ref{lem:B*_compact}, it follows that  for all $(h,k)\in\dR_+\times\Theta$,  $\phi\in   L^1(\rd a\otimes\nu)$ with $\NRM{\phi}_{ L^1(\rd a\otimes\nu)}\le 1$, we have $\ABS{\langle \lambda,T_*(t)(\phi)\rangle}\le\lambda_*\re^{\lambda_* t}$ and 
\begin{align*}
    \NRM{G(t)(\phi)_{(h,k)}-G(t)(\phi)}_{ L^1(\rd a\otimes\nu)}\le\lambda_*^2\re^{\lambda_* t}\int_0^t\NRM{v_{h,k}(t,s)-v(t,s)}_{ L^1(\rd a\otimes\nu)}\rd s,
\end{align*}
where $v(t,s)(a,\theta)=\ind_{a\ge t}u_*(a-t+s,\theta)$.
Since $u_*$ is a probability density on $\dR_+\times \Theta$, from \cite[Lemma~4.3]{brezis2011analyse}, for any $0\le s\le t$, we have
\begin{equation*}
    \NRM{v_{h,k}(t,s)-v(t,s)}_{ L^1(\rd a\otimes\nu)}\underset{(h,k)\to(0,0)}{\longrightarrow}0,
\end{equation*}
and then by the dominated convergence theorem, we deduce
\begin{equation}\label{Riez-1}
   \sup_{\phi:\NRM{\phi}_{ L^1(\rd a\otimes\nu)}\le 1}\NRM{G(t)(\phi)_{(h,k)}-G(t)(\phi)}_{ L^1(\rd a\otimes\nu)}\underset{(h,k)\to(0,0)}{\longrightarrow}0.
\end{equation}
Moreover, we have
\begin{align*}
\int_{\dR_+\times\Theta}\left|G(\phi)(a,\theta)\right|\rd a\nu(\rd\theta)&\le\int_0^t\left|\langle\lambda, T_*(s)(\phi)\rangle\right|\int_{\Theta}\int_{t}^{\infty}u_*(a-t+s,\theta)\rd a\nu(\rd\theta)\rd s\\
&=\int_0^t\left|\langle \lambda, T_*(s)(\phi)\rangle\right|\int_{\Theta}\int_{s}^{\infty}u_*(a,\theta)\rd a\nu(\rd\theta)\rd s\\
&\le\lambda_*\|\phi\|_{L^1(\rd a\otimes\nu)}te^{\lambda_* t},
\end{align*}
which implies that
\begin{equation}\label{Riez-2}
    \sup_{\phi:\NRM{\phi}_{ L^1(\rd a\otimes\nu)}\le 1}\int_r^\infty\int_\Theta|G(t)(\phi)(a,\theta)|\rd a\nu(\rd\theta)\underset{r\to \infty}{\longrightarrow}0.
\end{equation}

Hence from \eqref{Riez-1} and \eqref{Riez-2}, we conclude as in the proof of Lemma~\ref{lem:B*_compact}, by using Riez-Fréchet-Kolmogorov criterion in $L^1(\rd a\otimes\nu)$ and \cite[Corollary~4.27]{brezis2011analyse}, that $G(t)$ is a compact operator and thus $T^2_*(t)$.

\end{proof}

\bibliographystyle{abbrv}
\bibliography{biblio-article}

\end{document}